\documentclass{amsart}

\usepackage{amsmath}
\usepackage{amssymb, mathrsfs}
\usepackage{amsthm}
\usepackage[onehalfspacing]{setspace}
\usepackage{cite}

\newtheorem{Thm}{Theorem}[section]
\newtheorem{Lem}[Thm]{Lemma}
\newtheorem{Prop}[Thm]{Proposition}

\theoremstyle{definition}
\newtheorem{Def}[Thm]{Definition}

\theoremstyle{remark}
\newtheorem{Rmk}[Thm]{Remark}

\newcommand{\R}{\mathbb{R}}
\newcommand{\gexp}[2]{exp^{G}_{#1,#2}}
\newcommand{\gsexp}[2]{exp^{G^*}_{#1,#2}}
\newcommand{\g}{\mathfrak{g}}
\newcommand{\h}{\mathfrak{h}}
\newcommand{\gsub}[2]{\partial_G #1 (#2)}
\newcommand{\yth}{y_{\theta}}
\newcommand{\vth}{v_{\theta}}
\newcommand{\nbhd}[2]{\mathcal{N}_{#1} \left( #2 \right)}
\newcommand{\gsubs}[2]{\partial_G^* #1 (#2)}
\newcommand{\vol}[1]{\mathrm{Vol} \left( #1 \right)}
\newcommand{\dist}{\mathrm{dist}}
\newcommand{\len}[1]{\mathrm{Length} \left( #1 \right)}

\title[]{Local H$\ddot{\textrm{O}}$lder regularity of solutions to generated Jacobian equations}

\author{Seonghyeon Jeong}

\begin{document}

\maketitle

\begin{abstract}
Generated Jacobian equations are Monge-Amp$\grave{\textrm{e}}$re type equations which contain optimal transport as a special case. Therefore, optimal transport case has its own special structure which is not necessarily true for more general generated Jacobian equations. Hence the theory for optimal transport can not be directly transplanted to generated Jacobian equations. In this paper, we point out the difficulties that prevent applying the proof the local H$\ddot{\textrm{o}}$lder regularity of solutions of optimal transport problem from \cite{Loeper2009OnTR} directly to Generated Jacobian Equations, we then discuss how to handle these difficulties, and prove local H$\ddot{\textrm{o}}$lder regularity in the generated Jacobian equation case.
\end{abstract}

\section{Introduction}
In this paper, we will consider the H$\ddot{\textrm{o}}$lder regularity theory of solutions to generated Jacobian equations(GJE), which is one type of prescribed Jacobian equation(PJE). A (PJE) is a second order PDE of the form 
\begin{displaymath}
\label{PJE} \tag{PJE}
\det \left( D_x ( T(x, D\phi(x) ,\phi(x) ) ) \right) = \psi(x,D\phi(x),\phi(x)),
\end{displaymath}
where $T: X \times \R \times \R^n \to \R^n$ and $\psi: X \times \R \times \R$, with unkown $\phi : X \to \R$. If there are functions $G$ and $V$ that satisfy
\begin{displaymath}
\left\{ \begin{array}{rl}
D_x G(x, T(x,p,u), V(x,p,u)) & = p \\
G(x, T(x,p,u), V(x,p,u)) & = u
\end{array} \right.,
\end{displaymath}
then the $\eqref{PJE}$ can be written as follows :
\begin{equation}
\label{GJE} \tag{GJE}
\det \left( D^2_x \phi(x) - A(x, D\phi(x), \phi(x)) \right) = \bar{\psi}(x, D\phi(x), \phi(x))
\end{equation}
where
\begin{align*}
A(x,p,u) & = D_x^2G (x, T(x,p,u), V(x,p,u)) \\
\bar{\psi}(x,p,u) & = \det(E(x, T(x,p,u), V(x,p,u))) \psi(x,p,u)
\end{align*}
with $E$ from \eqref{Gnondeg} in 2.1. \eqref{GJE} is called a generated Jacobian equation and we call $G$ the generating function of $\eqref{GJE}$. Well-known examples of $\eqref{GJE}$ arise in Monge-Amp$\grave{\textrm{e}}$re equation, optimal transport, and near and far field reflector antenna design. The second boundary condition of $\eqref{GJE}$ is that for some given domain $Y \subset \R^n$, the image of $X$ under the map $T( \cdot, D\phi(\cdot), \phi(\cdot))$ is equal to $Y$. As a special case, if we have measures which are supported on $X$ and $Y$ (we call them the source and the target respectively), then this second boundary condition can be replaced by defining weak solutions using the measures. We discuss this in 2.3.\\
The main theorem of this paper is local H$\ddot{\textrm{o}}$lder regularity of solutions to $\eqref{GJE}$. In the literature, it is known that, if the source and the target are bounded away from 0 and $\infty$ with respect to the Lebesgue measure, and if the generating function satisfies (G3w) together with some other conditions, then the solution $\phi$ to the $\eqref{GJE}$ has local H$\ddot{\textrm{o}}$lder regularity(See \cite{Guillen2015PointwiseEA}). But in optimal transport, which is a special case of $\eqref{GJE}$, the author of \cite{Loeper2009OnTR} showed the local H$\ddot{\textrm{o}}$lder regularity using $\eqref{As}$ condition, but weaker assumption on the source measure. The main theorem of this paper is parallel to his result in optimal transport. \\
The idea of the proof of the main theorem of this paper will be imported from \cite{Loeper2009OnTR} and the arXiv version of \cite{Kim2007ContinuityCA}. But, since there are structural differences between $\eqref{GJE}$ and optimal transport, we need to adjust the idea in \cite{Loeper2009OnTR} to import it to $\eqref{GJE}$. One of the big differences is the dependency of the generating function on the scalar parameter $v$. In optimal transport case, the generating function is $G(x,y,v) = -c(x,y)-v$, so that by taking the derivative with respect to $x$ or $y$, dependency on $v$ vanishes. This is not necessarily true in more general $\eqref{GJE}$ case. Hence every estimate about the cost function $c(x,y)$ should be combined with an estimate about the scalar parameter $v$ to import them to $\eqref{GJE}$. Another big difference is that the conditions on optimal transport hold on the whole space $X \times Y \times \R$, and all the derivatives have bounded norm. More general generating functions, however, do not necessarily satisfy the conditions on the whole space $X \times Y \times \R$, but only on a subset $\g \subset X \times Y \times \R$. An example can be found in \cite{Liu2015OnTC}. Moreover, because of the dependency on the scalar parameter, the derivatives of the generating function might not be bounded. Hence whenever we want to use conditions and derivatives of the generating function, we need to check that the points used are in a ``nice'' set $\g$ and stay inside a compact set. Since we are going to focus on the local regularity, this can be done by localizing the argument. Lemma $\ref{localniceg}$ will allow us to choose points in a nice compact set.\\
We introduce some related results. Optimal transport case is done in \cite{Loeper2009OnTR} and the arXiv version of \cite{Kim2007ContinuityCA}. For other H$\ddot{\textrm{o}}$lder regularity in optimal transport, see \cite{Figalli2011HlderCA} and \cite{Guillen2012OnTL}. In \cite{Guillen2012OnTL}, the authors prove the H$\ddot{\textrm{o}}$lder regularity under a condition known as (QQ-conv), which is equivalent to (A3w) when the cost function is smooth, and the same authors prove analogous result for $\eqref{GJE}$ in \cite{Guillen2015PointwiseEA}. In reflector antenna design, some cases are done by various authors. See \cite{Gutirrez2014TheNF}, \cite{ABEDIN20161}, \cite{Gutirrez2019C1alphaestimatesFT}. For other regularity theory of $\eqref{GJE}$, see \cite{trudinger2012local}, \cite{Jhaveri2016PartialRO}, \cite{Guillen2015PointwiseEA}.
\section{Setting of the problem}

\subsection{Conditions for the generating function $G$}
To have the local holder regularity result, we need some structural conditions on the domains $X$ and $Y$ and on the generating function $G$. We first assume the regularity and monotonicity of the generating function : 
\begin{equation}
\label{Regular}
 G \in C^4(X \times Y \times \R) \tag{Regular}
\end{equation}
\begin{equation}
\label{Gmono}
 D_v G < 0 \tag{$G$-mono}
\end{equation}
From $\eqref{Gmono}$, we get a function $H : X \times Y \times \R \to \R$ that satisfies 
\begin{equation}
G(x,y,H(x,y,u)) = u.
\end{equation}
Note that the implicit function theorem ensures that $H \in C^4$ by $\eqref{Regular}$ and this, with $\eqref{Gmono}$, implies
\begin{equation}
\label{Hmono} \tag{$H$-mono}
D_u H < 0.
\end{equation}
As in optimal transport case, we need conditions like (A1), (A2)(sometimes these are called (twisted) and (non-deg) respectively) and (As) in \cite{Loeper2009OnTR}. But in some examples of $\eqref{GJE}$, the function $G$ does not satisfy these conditions on the whole domain $X \times Y \times \R$. Instead, there is a subset $\g \subset X \times Y \times \R$ on which the generating function $G$ satisfies conditions corresponding to (A1), (A2) and (As). Therefore, we assume that there is a subset $\g \subset X \times Y \times \R$ such that the following conditions hold :
\begin{equation}
\label{Gtwist}
(y,v) \mapsto (D_x G(x, \cdot , \cdot) , G(x,\cdot , \cdot)) \  \textrm{is injective on} \ \g_x \tag{$G$-twist}
\end{equation}
\begin{equation}
\label{G*twist}
x \mapsto -\frac{D_y G}{D_v G} (\cdot , y, v) \textrm{ is injective on } \g_{y,v} \tag{$G^*$-twist}
\end{equation}
\begin{equation}
\label{Gnondeg}
\det \left( D^2_{xy} G - D^2_{xv}G \otimes \frac{D_yG}{D_vG} \right) \neq 0 \  \textrm{on}\  \g \tag{$G$-nondeg}
\end{equation}
where $\g_x = \{ (y,v) | (x,y,v) \in \g \}$ and $\g_{y,v} = \{ x | (x,y,v) \in \g \}$. We will use $E$ to denote the matrix in the condition $\eqref{Gnondeg}$. Under these conditions, we can define the \emph{$G-$exponential function} on some subset of $T^*_xX$ and $T^*_yY$.

\begin{Def}
We define the $\gexp{x}{u}$ and $V_x$ by
\begin{displaymath}
\left\{ \begin{array}{rl}
D_xG(x, \gexp{x}{u}(p) , V_x(p,u)) & = p \\
G(x, \gexp{x}{u}(p) , V_x(p,u)) & = u
\end{array}
\right. .
\end{displaymath}
We call $\gexp{x}{u}$ the \emph{$G$-exponential function} with focus $(x,u)$. We define another exponential map $\gsexp{y}{v}$ by 
\begin{displaymath}
-\frac{D_y G}{D_v G} (\gsexp{y}{v}(q) , y, v) = q.
\end{displaymath}
We call $\gsexp{y}{v}$ the \emph{$G^*$-exponential function} with focus $(y,v)$.
\end{Def}
Note that by the implicit function theorem, the functions $\gexp{x}{u}, V_x(p,u)$ and $\gsexp{y}{v}$ are $C^3$ on the domain of each function. \\
In optimal transport case, the equations which are used to define the $c$-exponential maps and $c^*$-exponential maps are symmetric. In the above definition for the $G$-exponential function, the equations that we used to define $\gexp{x}{u}$ and $\gsexp{y}{v}$ do not look symmetric, but they actually play symmetric roles. See Remark 9.5 in \cite{Guillen2015PointwiseEA}.

\begin{Rmk}
\label{derivexp}
Let $I_{x,y} = \{ v \in \R | (x,y,v) \in \g \}$ and let $J_{x,y} = H(x,y,I_{x,y})$ and $ \h = \{ (x,y,u) | u \in J_{x,y} \}$. We will assume that $\h$ is open relative to $X \times Y \times \R$. Let $\h_{x,u} = \{ y \in Y | (x,y,u) \in \h \}$.
\begin{equation}
\label{Domopen}
\h \textrm{ is open relative to } X \times Y \times \R \tag{DomOpen}
\end{equation} 
We will denote the image of $\h_{x,u}$ under the map $D_xG(x, \cdot , H(x, \cdot, u))$ by $\h^*_{x,u}$. Then the $G$-exponential map $\gexp{x}{u}$ is defined on $\h^*_{x,u}$. Moreover, we get the expression $V_x(p,u) = H(x, \gexp{x}{u}(p) , u)$. With this expression, we can compute 
\begin{displaymath}
D_p \gexp{x}{u}(p) = E^{-1}(x, \gexp{x}{u}(p), V_x(p,u)).
\end{displaymath}
\end{Rmk}

To impose geometric conditions on the sets $X$ and $Y$, we define \emph{$G$-convexity} of sets.
\begin{Def}
\label{gconvex}
$X$ is said to be \emph{$G$-convex} if $\g_{y,v}^*$, which is the image of $\g_{y,v}$ under the map $\displaystyle -\frac{D_y G}{D_v G} (\cdot , y, v)$, is convex for any $(y,v) \in Y \times \R$. $Y$ is said to be \emph{$G^*$-convex} if $\h^*_{x,u}$ is convex for any $(x,u) \in X \times \R$.
\end{Def}
We also add convexity conditions to the supports of source and target.
\begin{equation}
\label{hDomConv}
X \textrm{ is } G\textrm{-convex} \tag{hDomConv}
\end{equation}
\begin{equation}
\label{vDomConv}
Y \textrm{ is } G^*\textrm{-convex} \tag{vDomconv}
\end{equation}

\begin{Def}
For $x \in X$ and $u \in \R$, let $y_0 , y_1 \in \h_{x,u}$ and let $p_i =D_x G(x, y_i, H(x, y_i, u))$. A \emph{$G$-segment} that connects $y_0$ and $y_1$ with focus $(x,u)$ is the image of $[p_0 , p_1]$ under the map $\gexp{x}{u}$
\begin{displaymath}
\{ \gexp{x}{u}((1-\theta)p_0 + \theta p_0) | \theta \in [0,1] \}.
\end{displaymath}
For $y \in Y$ and $v \in \R$, let $x_0, x_1 \in g_{y,v}$ and let $\displaystyle q_i =  -\frac{D_y G}{D_v G} (x_i , y, v)$. The \emph{$G^*$-segment} that connects $x_0$ and $x_1$ with focus $(y,v)$ is the image of $[q_0 , q_1]$ under the map $\gsexp{y}{v}$ : 
\begin{displaymath}
\{\gsexp{y}{v}((1-\theta)q_0 + \theta q_0)| \theta \in [0,1] \}.
\end{displaymath}
\end{Def}

\begin{Rmk}
The definition of $G$-convexity is different from the usual $G$-convex definition from the literature (See \cite{Guillen2015PointwiseEA}). The two definitions serve the same purpose, however, as these convexity conditions are used to ensure that $G$-segments are well defined. For example, if $y_0, y_1 \in \h_{x,u}$ then the $G$-segment that connects $y_0$ and $y_1$ with focus $(x,u)$ is well defined by Definition $\ref{gconvex}$. In \cite{Kim2007ContinuityCA}, the authors define convexity of domains using a subset $W$ of $M \times \bar{M}$(Definition 2.5). The convexity we have defined coincides with their horizontal and vertical convexity of $W$ when the generating function is given by a cost function.
\end{Rmk}

In optimal transport, there is an important condition called (A3w). This condition is a sign condition of a 4-tensor that was first introduced in \cite{Ma2005RegularityOP}, and \cite{Trudinger2006OnTS} (the tensor is called the MTW tensor sometimes). In optimal transport, the tensor is defined by $D^2_{pp} \left( - D^2_{xx}c(x, exp^c_x(p))\right)$, and it has the coordinate expression
\begin{displaymath}
MTW_{ijkl} = \left( c_{ij,r}c^{r,s}c_{s,pq} - c_{ij,pq} \right)c^{p,k}c^{q,l}
\end{displaymath}
(See, for example, \cite{Ma2005RegularityOP}). Then (A3w) condition assume that $MTW[\xi, \xi, \eta, \eta] \geq 0$ for any $\xi \perp \eta$. (A3w) is used to show many regularity results. In \cite{Figalli2011HlderCA}, the arXiv version of \cite{Kim2007ContinuityCA}, and \cite{Guillen2012OnTL}, (A3w) is used to show the H$\ddot{\textrm{o}}$lder regularity of potential functions. In \cite{Philippis2012SobolevRF}, (A3w) is used to show the Sobolev regularity of potential functions. In \cite{Loeper2009OnTR} However, the author uses the strengthened condition (As) to show the H$\ddot{\textrm{o}}$lder regularity result. 
\begin{equation}
\label{As} \tag{As} MTW[ \xi, \xi, \eta, \eta] \geq \delta | \xi |^2 | \eta |^2 \textrm{ whenever } \xi \perp \eta, \textrm{ for some } \delta>0.
\end{equation}
In optimal transport, $\eqref{As}$ is equivalent to  
\begin{equation}
\label{As'} \tag{As'} MTW [ \xi, \xi, \eta, \eta ] > 0 \textrm{ for } \xi \perp \eta.
\end{equation}
Obviously $\eqref{As}$ implies $\eqref{As'}$. For the other implication, note that $X \times Y$ is compact. Then $\eqref{Regular}$ and $\eqref{Gnondeg}$ implies that there exists $\delta >0$ such that $\eqref{As}$ is true with the $\delta$. Since optimal transport is a special case of (GJE), we should impose a condition that corresponds to the condition $\eqref{As}$ to import the idea from \cite{Loeper2009OnTR}. Therefore, we assume the following condition.
\begin{equation}
\label{G3s} \tag{G3s}  D^2_{pp} \mathcal{A}(x,p,u)[\xi, \xi, \eta, \eta ]  > 0 \textrm{ for any } \xi \perp \eta
\end{equation}
where $\mathcal{A}(x,p,u) =  D^2_{xx} G\left( x, \gexp{x}{u}(p), V_x(p,u) \right)$, $\xi \in T_x X$ and $\eta \in T^*_xX$. Here $\xi \perp \eta$ means that the dual pairing $\eta(\xi)$ is 0. Note that $p$ must be in $\h^*_{x,u}$ to well-define $\mathcal{A}(x,p,u)$. Here we point out that, unlike optimal transport, $\eqref{G3s}$ is not equivalent to, but weaker than the following.
\begin{displaymath}
D^2_{pp} \mathcal{A}(x,p,u)[\xi, \xi, \eta, \eta ]  > \delta |\xi|^2 |\eta|^2 \textrm{ for any } \xi \perp \eta \textrm{ for some } \delta >0.
\end{displaymath}
This is because $\delta$ depends on $(x,p,u) \in \h^* = \bigcup_{(x,u) \in X \times \R} \h^*_{x,u}$ and $\h^*$ is not compact in general. We use $\eqref{G3s}$ in this paper. Then some arguments that use $\delta$ in $\eqref{As}$ from \cite{Loeper2009OnTR} can not be applied. We resolve this problem by localizing the argument. See Remark $\ref{nicedom}$.

In this paper, we have that $X$ and $Y$ are subsets of $\R^n$. Then we can identify $T_xX$ and $T^*_xX$ with $\R^n$, and the dual pairing with usual inner product in $\R^n$. As such we will often use $\R^n$ for the tangent space and the cotangent space. (G3w) condition is the same condition $\eqref{G3s}$, but with $\geq$ instead of $>$.

\begin{Rmk}
The condition (A3w) is known to be equivalent to the next inequality
\begin{displaymath}
-c(y,\bar{x}(t)) + c(x, \bar{x}(t)) \leq \max \{ -c(y,\bar{x}(0)) + c(x, \bar{x}(0)), -c(y,\bar{x}(1)) + c(x, \bar{x}(1)) \}
\end{displaymath}
where $\bar{x}(t)$ is the $c$-segment that connects $\bar{x}(0)$ and $\bar{x}(1)$ with focus $x$ (See \cite{Loeper2009OnTR}, \cite{Villani2003TopicsIO}). This is called Loeper's property. This has the following geometric meaning : If two $c$-affine functions, which have $c-$subdifferentials $\bar{x}(0)$ and $\bar{x}(1)$ respectively, meet at $x$, then the $c$-affine functions that pass though that point with $c$-subdifferential $\bar{x}(t)$ lies under max of the two original $c$-affine functions. Therefore, this property is called geometric Loeper's property (gLp) sometimes. In \cite{Loeper2009OnTR}, he uses \eqref{As} to get (gLp) in some quantitative way. This can be done in the $G$-convex setting as well. We obtain a property corresponding to (gLp) for $G$-affine functions, and along with a quantitative version of (gLp) using $\eqref{G3s}$ in the lemma $\ref{qglp}$.
\end{Rmk}

\subsection{$G-$convex functions}
The solutions of \eqref{GJE} belong to a class known as \emph{$G$-convex function}. We define the $G$-convexity of a function by generalizing the usual convexity for the $c$-convexity in optimal transport case :

\begin{Def}
A function $\phi : X \to \R$ is said to be a \emph{$G$-convex function} if for any $x_0 \in X$ there exists $y_0 \in Y$ and $v_0 \in \R$ such that $(x_0 , y_0 , v_0 ) \in \g$ and 
\begin{align}
\label{gsubineq}
\phi(x_0) & = G(x_0, y_0, v_0 )   \\
\phi(x) & \geq G(x , y_0, v_0 ) \quad  \forall x \in X. \nonumber
\end{align}
We say that $(y_0,v_0)$ is a \emph{$G$-focus} of $\phi$ at $x_0$ and the function $G(x,y_0, v_0)$ is a \emph{$G-$supporting function} of $\phi$ at $x_0$ with focus $(y_0,v_0)$. We define the \emph{$G$-subdifferential} of $\phi$ at $x_0$ by 
\begin{displaymath}
\gsub{\phi}{x_0} = \{ y \in Y | (y,v) \textrm{ is a } G\textrm{-focus of } \phi \textrm{ at } x_0 \textrm{ for some } v \in \R\}
\end{displaymath}
and we define $ \displaystyle \gsub{\phi}{A} = \bigcup_{x \in A} \gsub{\phi}{x}$.
\end{Def}
With the definition above and Remark $\ref{derivexp}$, we can express $v_0$ as follows :
\begin{displaymath}
v_0 = H(x_0 , y_0 , \phi(x_0)).
\end{displaymath}
Therefore, the $G$-supporting function of a $G$-convex function $\phi$ at $x_0$ with $G$-sudifferential $y_0$ can be expressed as $G(x, y_0, H(x_0, y_0, \phi(x_0)))$. Also, note that $(x_0, y_0, v_0) \in \g$ is equivalent to $(x_0,y_0, \phi(x_0)) \in \h$. We will use this fact often. We call the function of the form $G( \cdot, y, v)$, \emph{$G$-affine function} with focus $(y,v)$, or simply a \emph{$G$-affince function}.

\begin{Rmk}
\label{rmkonunif}
The conditions on the generating function $G$ imply $G$-convexity of the $G$-subdifferential at a point which in turn implies the following proposition :

\begin{Prop}
\label{loctoglob}
If a $G$-affine function $G(\cdot , y_0 , v_0 )$ supports a $G$-convex function $\phi$ at $x_0$ locally, i.e.
\begin{displaymath}
\begin{array}{c}
\phi(x_0) = G(x_0, y_0, v_0) \\
\phi(x) \geq G(x, y_0, v_0) \textrm{ on some neighborhood of } x_0
\end{array}
\end{displaymath}
and if $(x_0, y_0, v_0) \in \g$, then $y_0 \in \gsub{\phi}{x_0}$.
\end{Prop}
\noindent
For example, it is proved in \cite{Guillen2015PointwiseEA} (Corollary 4.24), but under an extra condition on $J_{x,y}$, namely (unif) and the authors require the solution of the \eqref{GJE} to be ``nice''. 
\begin{flushleft}
\begin{tabular}{rl}
(unif) & $\exists \underline{u} , \overline{u}$ such that $[\underline{u} , \overline{u} ] \subset J_{x,y}$ for any $x \in X$ and $y \in Y$. \\
(nice) & The solution $\phi$ to the $\eqref{GJE}$  lies in $( \underline{u}, \overline{u} )$ :  $\phi(x) \in ( \underline{u}, \overline{u} ), \forall x \in X$. 
\end{tabular}
\end{flushleft}
In \cite{Guillen2015PointwiseEA}, the authors used these conditions for two reasons. First, they use these to check that the $G$-segments they are using are well-defined. Under these conditions, if $u \in ( \underline{u}, \overline{u} )$, then $(x,y, u) \in \h$ for any $x \in X$ and $y \in Y$ so that (DomConv*) assures that the $G$-segement for any pair of points in $Y$ with focus $(x,u)$ can be defined. Second, they get a compact set $X \times Y \times [ \underline{u}, \overline{u} ]$ that lies inside $\h$ with (unif). Then the norms of derivatives of the functions $G$ and $H$ will be bounded, and will be bounded away from 0 if it is signed. For the local regularity, we do not need the fixed interval $[ \underline{u}, \overline{u} ]$ inside $J_{x,y}$ for any $x \in X$ and $y \in Y$. Instead, we will use the following weaker conditions :
\begin{flushleft}
\begin{tabular}{rp{10cm}}
(unifw) & $\exists a , b : X \times Y \to \R$ which are continuous and $[a(x,y) , b(x,y)] \subset J_{x,y}$ for any $x \in X$ and $y \in Y$. \\
(nicew) & The solution $\phi$ to the (GJE) lies in $(a, b)$ : $\phi(x) \in (a(x,y) , b(x,y))$ for any $x \in X$ and $y \in \gsub{\phi}{x}$.
\end{tabular}
\end{flushleft}
Note that the $G$-segments used in \cite{Guillen2015PointwiseEA} to show Proposition $\ref{loctoglob}$ can be well-defined by $\eqref{hDomConv}$ and $\eqref{vDomConv}$. In fact, the $G$-segments constructed in \cite{Guillen2015PointwiseEA} connect two points in $\gsub{\phi}{x} \subset \h_{x, \phi(x)}$. Hence $\eqref{vDomConv}$ is enough to define the $G$-segments. With similar reasoning, $\eqref{hDomConv}$ ensures that the $G^*$-segments that are used are well-defined. Moreover, the sets 
\begin{equation}
\begin{array}{c}
\Phi := \{ (x,y,u) | u \in [a(x,y),b(x,y)] \} \subset \h \\
\Psi := \{ (x,y,v) | v \in H(x,y,[a(x,y),b(x,y)]) \}
\end{array}
\end{equation}
are compact. Note that $\eqref{hDomConv}$ and $\eqref{vDomConv}$ do not imply that $(x,\yth, u) \in \Phi$, where $\yth$ is the $G-$segment connecting $y_0$ and $y_1$ with focus $(x,u)$, so we can not bound norms of $G$ and its derivatives on a $G$-segment only using $\Phi$. But we know that $\Phi$ lies in a compact set $X \times Y \times [ \min a , \max b ]$, and $\Psi$ lies in the corresponding compact set. Therefore, we use norms on these compact sets that contain $\Phi$ and $\Psi$ and we still can use the same proof of the proposition $\ref{loctoglob}$ with assumption (unifw) and (nicew) instead of (unif) and (nice).
\end{Rmk}

\begin{Rmk}
\label{nicedom}
For a compact subset $S$ of $\h$, the condition \eqref{Gnondeg} implies that we have a constant $C_e$ that depends on $S$ such that
\begin{equation}
\label{Ce1}
\frac{1}{C_e} \leq \|E\| \leq C_e 
\end{equation}
where $\|E\|$ is the operator norm of $E$. This with Remark $\ref{derivexp}$ implies that
\begin{equation}
\label{Ce2}
\frac{1}{C_e} |p_1 - p_0| \leq |\gexp{x}{u}(p_1) - \gexp{x}{u}(p_0)| \leq C_e |p_1 - p_0|
\end{equation}
when $(x, \gexp{x}{u}((1-\theta)p_0 + \theta p_1) , u) \in S, \forall \theta \in [0,1]$. Also, the (G3s) condition implies that we have a constant $\alpha$ that depends on $S$ such that
\begin{equation}
\label{alpha}
D^2_{pp} \mathcal{A} (x,p,u) [\xi, \xi, \eta, \eta ] > \alpha
\end{equation}
for any unit $\xi$ and $\eta$ such that $\xi \perp \eta$ if $(x, \gexp{x}{u}(p), u) \in S$. Note that by tensoriality, we have
\begin{equation}
D^2_{pp} \mathcal{A} (x,p,u) [\xi, \xi, \eta, \eta ] > \alpha|\xi|^2|\eta|^2
\end{equation}
for non unit $\xi$ and $\eta$.
Moreover, compactness of $X \times Y \times [\min a, \max b]$ implies that we have $\beta >0$ such that
\begin{equation}
\label{beta}
D_v G < - \beta \textrm{ on } X \times Y \times [\min a, \max b].
\end{equation}
\end{Rmk}

With these conditions, we show some simple propositions for the $G-$subdifferential of $\phi$.  We will denote the $r$ neighborhood of a set $A$ by $\nbhd{r}{A}$.

\begin{Prop}
\label{compgsub}
The subdifferential of $\phi$ at a point $x$ is a closed subset of $Y$ and it lies compactly in $\h_{x, \phi(x)}$ : 
\begin{displaymath}
\gsub{\phi}{x} \Subset \h_{x, \phi(x)} \subset Y.
\end{displaymath}
\end{Prop}
Note that Proposition $\ref{compgsub}$ for optimal transport case is much easier to prove because $J_{x,y} =\R$ for any $(x,y) \in X \times Y$ i.e. $\h = X \times Y \times \R$ so that we only need to check the inequality ($\ref{gsubineq}$). But in \eqref{GJE} case, compactness is not trivial since showing the inequality ($\ref{gsubineq}$) is not enough, and we need to check $(x, y, \phi(x)) \in \h$.
\begin{proof}
First we show that $\gsub{\phi}{x}$ is closed. Suppose $y \in \overline{\gsub{\phi}{x}}$. Then there exists a sequence $y_i \in \gsub{\phi}{x}$ that converges to $y$. Then from (unifw) and (nicew), we have 
\begin{displaymath}
\begin{array}{rl}
\phi(x) & \in \bigcap_{i=1}^{\infty} (a(x, y_i) , b(x, y_i)) \\
& \subset [ \sup a(x, y_i), \inf b(x, y_i) ] \\
& \displaystyle \subset \left[ \lim_{i \to \infty} a(x, y_i), \lim_{i \to \infty} b(x, y_i) \right] \\
& = [a(x,y), b(x,y)].
\end{array}
\end{displaymath}
This implies that $(x,y,\phi(x)) \in \h$ by (unifw). Moreover, from the definition of $G-$subdifferential, we have
\begin{displaymath}
G(z, y_i, H(x,y_i, \phi(x))) \leq \phi(z), \ \forall z \in X.
\end{displaymath}
Taking $i \to \infty$, we get
\begin{displaymath}
G(z, y, H(x, y, \phi(x))) \leq \phi(z), \ \forall z \in X.
\end{displaymath}
This inequality with $(x,y,\phi(x)) \in \h$ shows that $y \in \gsub{\phi}{x}$. Therefore $\gsub{\phi}{x}$ is closed, and hence compact as a closed subset of $Y$. Now note that $\gsub{\phi}{x} \subset \h_{x, \phi(x)}$. Moreover, from the openness assumption of $\h$ (Remark $\ref{derivexp}$), $\h_{x,\phi(x)}$ is relatively open with respect to $Y$. Therefore compactness of $\gsub{\phi}{x}$ shows that $\gsub{\phi}{x} \Subset \h_{x, \phi(x)}$.
\end{proof}

The next proposition is about continuity of the $G$-subdifferential of a $G$-convex function. Note that the $G$-subdifferential is a set valued function, so that the continuity here is not the ordinary continuity of a single valued function. But if the $G$-subdifferential is single valued, the next proposition actually coincides with the continuity of a single valued function.

\begin{Prop}
\label{gsubconti}
Let $\phi$ be a $G$-convex function with (nicew). Let $x \in X$. Then for $\epsilon >0$, there exists $\delta >0$ such that if $|z - x| \leq \delta$ then
\begin{displaymath}
\gsub{\phi}{z} \subset \nbhd{\epsilon}{\gsub{\phi}{x}}.
\end{displaymath}
\end{Prop} 
\begin{proof}
Suppose it is not true. Then we get a sequence $x_k \in X$ and $y_k \in \gsub{\phi}{x_k}$ such that $x_k \to x$ as $k \to \infty$ but $y_k \notin \nbhd{\epsilon}{\gsub{\phi}{x}}$ for any $k$. Since $Y$ is compact, we can extract a subsequence such that $y_k \to y$. Note that by (nicew), we have $(x,y,\phi(x)) \in {\Phi} \subset \h$. From the choice of $y_k$, we have $y \notin \nbhd{\epsilon}{\gsub{\phi}{x}}$. But since $y_k$ are in subdifferentials of $\phi$, we have
\begin{equation}
\label{prop3.4:1}
\phi(z) \geq G(z , y_k, H(x_k, y_k, \phi(x_k))), \ \forall z \in X.
\end{equation}
We can take limit on ($\ref{prop3.4:1}$) because $\phi$, $G$, and $H$ are continuous. Hence we get
\begin{displaymath}
\phi(z) \geq G(z,y,H(x,y,\phi(x)))
\end{displaymath}
which implies $y \in \gsub{\phi}{x}$, which contradicts to $y \notin \nbhd{\epsilon}{\gsub{\phi}{x}}$.
\end{proof}

\subsection{weak solutions to the (GJE)}
Let $\mu$ be the source measure, a probability measure supported on $X$, and let $\nu$ be the target measure, a probability measure supported on $Y$. Then we can interpret the solutions to the (GJE) with second boundary condition in the following ways, which are analogies of the Alexandrov solution and the Brenier solution in optimal transport case.

\begin{Def}
Let $\phi : X \to \R$ be a $G$-convex function. Then \\
1. $\phi$ is called a \emph{weak Alexandrov solution} to the (GJE) if
\begin{displaymath}
\mu(A) = \nu(\gsub{\phi}{A}), \ \forall A \subset X.
\end{displaymath}
2. $\phi$ is called a \emph{weak Brenier solution} to the (GJE) if
\begin{displaymath}
\nu(B) = \mu(\partial_G^{-1}\phi (B) ), \ \forall B \subset Y.
\end{displaymath}
\end{Def}

Note that the $G$-subdifferential of $\phi(x)$ is single valued $\mu-$a.e. in $X$ as it is semi-convex by Proposition $\ref{semiconvex}$. Hence a weak Brenier solution satisfies the push-forward condition $\partial_G \phi _\sharp \mu_0 = \mu$. In optimal transport, it is well known that the Brenier solution is not necessarily an Alexandrov solution. To prevent this, we need some convexity condition on the domains. See 4.6 in \cite{Figalli2017TheME} for the Monge-Amp$\grave{\textrm{e}}$re case, \cite{Ma2005RegularityOP} for the general optimal transport case. To develop the local regularity theory in this paper, we will use a weak Alexandrov solution.

\subsection{conditions for measures and main result}
We will assume the target measure $\nu$ is bounded away from 0 and $\infty$ with respect to the Lebesgue measure on $Y$. For the local regularity theory, we will assume one of the following assumptions : \\
\begin{tabular}{c p{10cm}}
1. & $\exists p \in (n,\infty]$ and $C_{\mu}$ such that for any $x \in X$ and $r \geq 0$ we have $\mu( B_r(x) ) \leq C_{\mu} r^{n(1-\frac{1}{p})}$. \\
2. & $\exists f : \R^{+} \to \R^{+}$ such that $\lim_{r \to 0} f(r) = 0$ and for any $x \in X$ and $r \geq 0$ we have $\mu(B_r(x))\leq f(r)r^{n(1-\frac{1}{n})}$.
\end{tabular}\\
Then the main theorem of this paper is the following

\begin{Thm}
\label{main}
Suppose $X$ and $Y$ are compact domains in $\R^n$ and let $G : X \times Y \times \R \to \R$ be the generating function. Let $\mu$ and $\nu$ be probability measures on $X$ and $Y$ respectively. Assume that $G$ satisfies \eqref{Regular}, \eqref{Gmono}, \eqref{Gtwist}, \eqref{G*twist}, \eqref{Gnondeg}, \eqref{G3s}, and (unifw). Assume also that $X$ and $Y$ satisfy \eqref{hDomConv} and \eqref{vDomConv} and the target measure $\nu$ is bounded away from 0 and $\infty$ with respect to the Lebesgue measure on $Y$. Let $\phi$ be an weak Alexandrov solution to the equation ($\ref{GJE}$) that satisfies (nicew). Then we have the followings : \\
\begin{tabular}{c p{10cm}}
1. & If there exist $p \in (n, \infty]$ and $C_{\mu}$ such that $\mu( B_r(x)) \leq C_{\mu} r^{n(1-\frac{1}{p})}$ for all $r \geq 0$, $x \in X$, then $\phi \in C^{1,\sigma}_{loc}(X)$. \\
2. & If there exist $f : \R^{+} \to \R^{+}$ such that $\lim_{r \to 0} f(r) = 0$ and $\mu(B_r(x)) \leq f(r)r^{n(1-\frac{1}{n})}$ for all $r \geq 0$, $x \in X$, then $\phi \in C^{1}_{loc}(X)$.
\end{tabular}\\
Here, $\rho = 1- \frac{n}{p}$ and $\sigma = \frac{\rho}{4n-2+\rho}$.
\end{Thm} 
Note that the exponents that appears in above theorem are the same as in \cite{Loeper2009OnTR}.

\section{Proof of local holder regularity}
For the proof of local H$\ddot{\textrm{o}}$lder regularity, we will do most computations in the set $X \times Y \times [\min a , \max b]$. Since this set is compact we will be able to get finite quantities in each computation with a localizing argument. Then we will be able to apply the idea from \cite{Loeper2009OnTR} for the proof of each lemma. The main difference  between \cite{Loeper2009OnTR} and this paper comes from the structural difference of a cost function and a generating function. In particular, a generating function has its own nice subdomain where the structural conditions hold true, whereas the corresponding conditions hold on the whole domain in optimal transport case. Therefore, we need to check that the points at which we use the conditions are in the nice subdomain. Moreover, each derivative of $G$ still depends on the scalar parameter $v$, hence we need to take care of extra terms that come from the dependency on the scalar parameter $v$. In addition, estimates on the cost function $c(x,y)$ should be done on $G(x,y,v)$. \\
Through out this paper, we will use tensor notation for derivatives many times. For example, we view $D^2_{xx}G$ as a 2-tensor, and we use square bracket ``[ , ]'' for tensor notation.
\begin{displaymath}
D^2_{xx}G [\xi, \xi] = \xi^t D^2_{xx}G \xi.
\end{displaymath}

\subsection{Quantitative (glp) with (G3s)}
We start with some estimation on $G$-affine functions. In this subsection, $x_m$ is a point in $X$, $u \in \R$, and $y_0, y_1 \in \h_{x_m, u}$. Also, for $\theta \in [0,1]$, we denote the $G$-segment that connects $y_0$ and $y_1$ with focus $(x_m,u)$ by $\yth$, and we use $\vth = H(x_m, \yth, u)$ and $p_i = D_xG(x_m,y_i,v_i)$. Moreover, we will assume that the points $(x_m, \yth, u) \in S$ for some compact set $S \Subset \h$. Then by remark$\ref{nicedom}$, we get constants $C_e$ and $\alpha$ that depend on $S$ and satisfy ($\ref{Ce1}$),($\ref{Ce2}$), and ($\ref{alpha}$).

\begin{Lem}
\label{2dest}
For some constant $C_1$ that depends on the $C^3$ norm of $G$, $C^1$ norm of $H$, and $C_e$, we have
\begin{equation}
\left| \left( D^2_{xx}G(x_m,y_{\theta},v_{\theta}) - D^2_{xx}G(x_m,y_{\theta'},v_{\theta'})\right) [\xi, \xi] \right| \leq C_1|\theta - \theta'||p_1 - p_0||\xi|^2
\end{equation}
\end{Lem}
\begin{proof}
\begin{displaymath}
\begin{array}{l}
\|D^2_{xx} G(x_m, y_{\theta} , H(x_m, y_{\theta}, u)) - D^2_{xx} G(x_m, y_{\theta'}, H(x_m, y_{\theta'}, u))\| \\
\leq \|D^3_{xxy}G\||y_{\theta} - y_{\theta'}| + \|D^3_{xxv}G\| \|D_yH\| |y_{\theta} - y_{\theta'}| \\
\leq (\|D^3_{xxy}G\| + \|D^3_{xxv}G\| \|D_yH\|)C_e |\theta - \theta'||p_1 - p_0|.
\end{array}
\end{displaymath}
We set $C_1 = (\|D^3_{xxy}G\| + \|D^3_{xxv}G\| \|D_yH\|)C_e$.
\end{proof}

\begin{Lem}
\label{2dest2}
Let $\xi_p = \mathrm{Proj}_{p_1 - p_0} (\xi)$, where $\mathrm{Proj}_{p}$ is the orthogonal projection onto $p$. Then for some constants $\Delta_1$ and $\Delta_2$ that depend on $\alpha$, the $C^4$ norm of $G$, we have
\begin{displaymath}
D^2_{xx}G(x_m,y_{\theta}, v_{\theta})[\xi, \xi] \leq  \begin{array}{l}
\left((1-\theta) D^2_{xx}G(x_m, y_0, v_0) + \theta D^2_{xx}G(x_m, y_1, v_1)\right)[\xi, \xi] \\
+\theta(1-\theta)|p_1 - p_0|^2(-\Delta_1 |\xi|^2 + \Delta_2 |\xi_p|^2).
\end{array}
\end{displaymath}
\end{Lem}
\begin{proof}
Let $f_{\xi} : [0,1] \to \R$ be such that
\begin{displaymath}
f_{\xi}(\theta) = D^2_{xx}G(x_m, y_{\theta}, v_{\theta})[\xi, \xi].
\end{displaymath}
Let $\xi' = \xi - \xi_p $ so that $\xi' \perp \xi_p$. Then we can apply \eqref{G3s}, and we obtain
\begin{displaymath}
f_{\xi'}'' \geq \alpha |p_1 - p_0|^2 |\xi'|^2
\end{displaymath}
and from this uniform convexity, we have
\begin{equation}
\label{lem3.2:1}
f_{\xi'} (\theta) \leq \theta f_{\xi'}(1) + (1-\theta)f_{\xi'}(0) - \frac{1}{2} \alpha |p_1 - p_0|^2|\xi'|^2 \theta(1-\theta).
\end{equation}
Let $g_{\xi} = f_{\xi} - f_{\xi'}$. Then 
\begin{align*}
g_{\xi}''(\theta) & = f_{\xi}''(\theta) - f_{\xi'}'' (\theta)\\
& = D^2_{pp} \mathcal{A}[ \xi, \xi, p_1 - p_0, p_1 - p_0] - D^2_{pp}  \mathcal{A} [ \xi', \xi', p_1 - p_0, p_1 - p_0] \\
& = 2D^2_{pp} \mathcal{A} [\xi', \xi_p, p_1 - p_0, p_1 - p_0] + D^2_{pp} \mathcal{A} [ \xi_p, \xi_p, p_1 - p_0, p_1 - p_0].
\end{align*}
Therefore, bounding $|\xi'|$ by $|\xi|$, we obtain
\begin{displaymath}
|g_{\xi}''| \leq 3 \|D^2_{pp}\mathcal{A} \| |p_1 - p_0|^2|\xi||\xi_p|
\end{displaymath}
and from this bound, we have
\begin{equation}
\label{lem3.2:2}
g_{\xi}(\theta) \leq \theta g_{\xi}(1) + (1-\theta) g_{\xi}(0) + \frac{3}{2}\|D^2_{pp}\mathcal{A}\||p_1 - p_0|^2|\xi||\xi_p| \theta (1-\theta).
\end{equation}
Combining ($\ref{lem3.2:1}$) and ($\ref{lem3.2:2}$), we obtain
\begingroup
\allowdisplaybreaks
\begin{align*}
D^2_{xx}G(x_m,y_{\theta}, v_{\theta}) & =  g_{\xi} + f_{\xi'} \\
&  \leq \theta g(1) + (1-\theta) g(0) + \frac{3}{2} \|D^2_{pp}\mathcal{A}\||p_1 - p_0|^2|\xi||\xi_p| \theta (1-\theta) \\
&  \quad + \theta f_{\xi'}(1) + (1-\theta)f_{\xi'}(0) - \frac{1}{2} \alpha |p_1 - p_0|^2|\xi'|^2 \theta(1-\theta) \\
& = \theta D^2_{xx}G(x_m,y_1, u) [\xi, \xi] + (1-\theta)D^2_{xx}G(x_m,y_0, u)(\xi, \xi) \\
&  \quad + \theta(1-\theta) | p_1 - p_0|^2 \left( -\frac{\alpha}{2}|\xi'|^2 + \frac{3}{2}\|D^2_{pp}\mathcal{A}\| |\xi||\xi_p| \right) \\
& \leq \theta D^2_{xx}G(x_m,y_1, u) [\xi, \xi] + (1-\theta)D^2_{xx}G(x_m,y_0, u)(\xi, \xi) \\
& \quad + \theta(1-\theta) | p_1 - p_0|^2 \left( -\frac{\alpha}{2}|\xi|^2 + (\frac{3}{2}\|D^2_{pp}\mathcal{A}\| +\alpha)|\xi||\xi_p| \right).
\end{align*}
\endgroup
Here, we use weighted Young's inequality
\begin{displaymath}
(\frac{3}{2}\|D^2_{pp}\mathcal{A}\|+\alpha) |\xi||\xi_p| \leq \frac{\alpha}{4}|\xi|^2 + \alpha^{-1} (\frac{3}{2}\|D^2_{pp}\mathcal{A}\|+\alpha)^2 |\xi_p|^2.
\end{displaymath}
Then we obtain
\begin{align*}
D^2_{xx}G(x_m, \yth, \vth) & \leq \theta D^2_{xx}G(x_m,y_1, u) [\xi, \xi] + (1-\theta)D^2_{xx}G(x_m,y_0, u)[\xi, \xi] \\
& + \theta(1-\theta) | p_1 - p_0|^2 \left( -\frac{\alpha}{4} |\xi|^2 + \alpha^{-1} (\frac{3}{2}\|D^2_{pp}\mathcal{A}\| +\alpha)^2 |\xi_p|^2 \right).
\end{align*}
Hence we get the inequality with $\Delta_1 = \frac{\alpha}{4} $ and $\Delta_2  = \alpha^{-1}(\frac{3}{2} \|D^2_{pp}\mathcal{A}\| +\alpha)^2$.
\end{proof}

The next lemma is the quantitative version of (gLp). We will use the (G3s) condition through Lemma $\ref{qglp}$ later.

\begin{Lem}
\label{qglp}
Define $\bar{\phi}(x) : X \to \R$ by
\begin{displaymath}
\bar{\phi}(x) = \max \{ G(x, y_0, v_0), G(x,y_1, v_1) \}
\end{displaymath}
Then we have the quantitative (gLp) : 
\begin{equation}
\bar{\phi}(x) \geq  G(x, \yth, \vth) + \delta_0\theta(1-\theta)|y_1 - y_0|^2|x-x_m|^2 -\gamma |x-x_m|^3
\end{equation}
\label{qglp1}
for $\theta \in [ \epsilon , 1- \epsilon]$ and $|x - x_m| \leq C \epsilon$. 
\end{Lem}
\begin{proof}
Note that by taking the Taylor series, 
\begin{displaymath}
G(x,y_i, v_i) = u + \langle D_xG(x_m,y_i,v_i), (x-x_m) \rangle + \frac{1}{2} D^2_{xx}G(x,y_i,v_i)[x-x_m, x-x_m] + o(|x-x_m|^2).
\end{displaymath}
Therefore, we have
\begin{align*}
\bar{\phi}(x) & \geq \theta G(x,y_0,v_0) + (1 - \theta) G(x,y_1,v_1) \\
& = u + \langle \theta p_1 + (1-\theta) p_0 , x-x_m \rangle \\
&\ + \frac{1}{2} \left( \theta D^2_{xx}G(x, y_0, v_0) + (1-\theta) D^2_{xx}(x, y_1, u_1)\right)[x-x_m , x-x_m] + o(|x-x_m|^2).
\end{align*}
Applying Lemma $\ref{2dest2}$, we obtain
\begin{align}
\label{qglp2}
\bar{\phi}(x) & \geq u + \langle \theta p_1 + (1-\theta)p_0, x-x_m \rangle + \frac{1}{2} D^2_{xx}G(x_m, y_{\theta}, \vth) [x-x_m, x-x_m] \\
& \ \ - \frac{1}{2}\theta(1-\theta)|p_1 - p_0|^2 ( - \Delta_1 |x-x_m|^2 + \Delta_2 |(x-x_m)_p|^2) + o(|x-x_m|^2) \nonumber
\end{align}
for any $\theta \in [0,1]$. Let $\theta' \in [0,1]$, then we can write ($\ref{qglp2}$) with $\theta'$. Let us call this inequality ($\ref{qglp2}$'). Then adding and subtracting the right hand side of ($\ref{qglp2}$) to the right hand side of ($\ref{qglp2}$') and reordering some terms, we get
\begin{align}
\label{qglp3}
\bar{\phi}(x) & \geq u + \langle \theta p_1 + (1-\theta)p_0,x-x_m \rangle + \frac{1}{2} D^2_{xx}(x_m, y_{\theta}, \vth) [x-x_m, x-x_m] \nonumber \\
& \quad + \frac{1}{2} \Delta_1 \theta(1-\theta)|p_1 - p_0|^2  |x-x_m|^2 \nonumber \\
& \quad +(\theta' - \theta)\langle p_1 - p_0, x - x_m \rangle -\frac{1}{2}\theta(1-\theta)\Delta_2|p_1 - p_0|^2|(x - x_m)_p|^2 \nonumber \\
& \quad + \frac{1}{2} \left( D^2_{xx}G(x_m, y_{\theta'},v_{\theta'}) - D^2_{xx}G(x_m, y_{\theta},\vth) \right) [x-x_m, x-x_m] \\
& \quad + \frac{1}{2}\Delta_1 \left( (\theta'(1-\theta') - \theta(1-\theta) \right)|p_1 - p_0|^2|x-x_m|^2 \nonumber \\
& \quad + \frac{1}{2}\Delta_2 \left( (\theta(1-\theta) - \theta'(1-\theta') \right)|p_1 - p_0|^2|(x-x_m)_p|^2 + o(|x-x_m|^2) \nonumber
\end{align}
Note that by definition of $(x-x_m)_p$, we have $|p_1 - p_0||(x - x_m)_p| = | \langle p_1 - p_0,x - x_m \rangle|$. Hence, we can write the third line $L_3$ as follows
\begin{displaymath}
L_3 = [\theta' - \theta -\frac{1}{2}\theta(1-\theta)\Delta_2 \langle p_1 - p_0, x - x_m \rangle]\langle p_1 - p_0 ,x - x_m \rangle.
\end{displaymath}
Therefore, if we choose 
\begin{equation}
\label{theta'}
\theta' = \theta + \theta(1-\theta) \Delta_2 \langle p_1 - p_0, x - x_m \rangle ,
\end{equation}
we can make $L_3 = 0$. To ensure $\theta'$ is in $[0,1]$, we first assume that $\theta$ is away from 0 and 1, i.e. we assume $\theta \in [\epsilon, 1- \epsilon]$ for $\epsilon >0$. Then we can make the second term $\theta(1-\theta) \Delta_2 \langle p_1 - p_0, x - x_m \rangle$ small. by assuming
\begin{displaymath}
|x - x_m| \leq \frac{4 \epsilon}{ \Delta_2 |p_1 - p_0|} \leq \frac{\epsilon}{\theta(1-\theta) \Delta_2 |p_1 - p_0|}.
\end{displaymath}
Under these assumptions and ($\ref{theta'}$), we get $\theta' \in [0,1]$ and $L_3 = 0$. We can apply Lemma $\ref{2dest}$ to the forth line $L_4$ of ($\ref{qglp3}$) to get
\begin{align}
\label{qglp4}
L_4 & = \frac{1}{2} \left( D^2_{xx}G(x_m, y_{\theta'},v_{\theta'}) - D^2_{xx}G(x_m, y_{\theta},\vth) \right) [x-x_m, x-x_m] \nonumber \\ 
& \geq - C_1 |\theta - \theta'||p_1 - p_0||x - x_m|^2 \\
& \geq -C_1 \theta(1-\theta)\Delta_2 |p_1 - p_0|^2|x-x_m|^3 \nonumber \\
& \geq - \frac{1}{4} C_1 \Delta_2 |p_1 - p_0|^2|x-x_m|^3 \nonumber
\end{align}
For the fifth and sixth line, $L_5$ and $L_6$, note that by ($\ref{theta'}$),
\begin{align*}
\theta'(1-\theta') - \theta(1-\theta) & = (\theta - \theta')(\theta + \theta' -1) \\
& = - \theta(1-\theta)\Delta_2 \langle p_1 - p_0, x-x_m \rangle (\theta + \theta' -1)
\end{align*}
so that we can bound
\begin{align}
\label{qglp5}
|L_5|  & = \left| \Delta_1 [(\theta'(1-\theta') - \theta(1-\theta)]|p_1 - p_0|^2|x-x_m|^2 \right| \nonumber \\
& \leq \theta(1-\theta)(\theta+\theta'-1)\Delta_1 \Delta_2|p_1 - p_0|^3|x-x_m|^3 \\
& \leq \frac{1}{4} \Delta_1 \Delta_2|p_1 - p_0|^3|x-x_m|^3 \nonumber
\end{align}
\begin{align}
\label{qglp6}
|L_6| & =\left| \Delta_2 [(\theta(1-\theta) - \theta'(1-\theta')]|p_1 - p_0|^2|(x-x_m)_p|^2 \right| \nonumber \\
& \leq \theta(1-\theta)(\theta+\theta'-1)(\Delta_2)^2 |p_1 - p_0|^3|x-x_m|^3 \\
& \leq \frac{1}{4} (\Delta_2)^2  |p_1 - p_0|^3|x-x_m|^3. \nonumber
\end{align}
Combining ($\ref{qglp4}$), ($\ref{qglp5}$), and ($\ref{qglp6}$), we can bound ($\ref{qglp3}$) from below 
\begin{align}
\label{qglp7}
\bar{\phi}(x) \geq & u + \langle \theta p_1 + (1-\theta)p_0, x-x_0 \rangle + \frac{1}{2} D^2_{xx}G(x_m, y_{\theta}, \vth) [x-x_m, x-x_m]  \nonumber\\
& + \Delta_1 \theta(1-\theta)|p_1 - p_0|^2  |x-x_m|^2 \\
& - C_2(|p_1 - p_0|^2 + |p_1 - p_0|^3)|x-x_m|^3 + o(|x-x_m|^2) \nonumber
\end{align}
where $C_2$ depends on $C_1$, $\Delta_1$, and $\Delta_2$. We apply Taylor's Theorem to the first line of ($\ref{qglp7}$) and change it to $G(x,\yth,\vth)$ with $o( |x-x_m|^2)$. Moreover, the little o term $o(|x-x_m|^2)$ is at least $O(|x-x_m|^3)$ because the generating function is $C^4$. Therefore we can put it with the $|x-x_m|^3$ term, and we get
\begin{align*}
\bar{\phi}(x) \geq &  G(x,\yth,\vth) + \Delta_1 \theta(1-\theta)|p_1 - p_0|^2  |x-x_m|^2 \\
& -C_2(1 + |p_1 - p_0|^2 + |p_1 - p_0|^3)|x-x_m|^3
\end{align*}
possibly taking larger $C_2$ then before. Finally, we bound $|p_1 - p_0|$ by $C_e \mathrm{diam}(Y)$ and $\frac{1}{C_e}|y_1 - y_0|$ from above and below to get
\begin{displaymath}
\bar{\phi}(x) \geq  G(x,\yth,\vth) + \frac{\Delta_1}{C_e^2}\theta(1-\theta)|y_1 - y_0|^2|x-x_m|^2 -\gamma |x-x_m|^3.
\end{displaymath}
Hence we obtain the lemma with $\delta_0 = \Delta_1 / C_e^2$ and $\gamma = C_2(1+C_e^2 \mathrm{diam}(Y)^2 + C_e^3 \mathrm{diam}(Y)^3) $.
\end{proof}

\subsection{Local estimates for $G-$convex functions}

If a $G$-convex function $\phi$ is $C^2$, then for any $x \in X$ we get a $G$-supporting function at $x$. From the definition of a $G$-supporting function, the difference of $\phi$ and the $G$-supporting function attains global minimum at $x$ and the regularity condition implies
\begin{displaymath}
D^2_{xx} \phi \geq \ D^2_{xx}G \geq - \| D^2_{xx} G \| I.
\end{displaymath}
Hence it is semi convex. We show the semi convexity without $C^2$ assumption on the $G$- convex function in the following proposition.

\begin{Prop}
\label{semiconvex}
Let $\phi$ be a $G$-convex function that satisfies (nicew). Then we have following inequality
\begin{equation}
\phi(x_t) \leq (1-t)\phi(x_0) + t\phi(x_1) +\frac{1}{2}t(1-t)\| D^2_{xx} G \||x_0 - x_1|^2
\end{equation}
where $x_t = (1-t) x_0 + t x_1$. In particular, $\phi$ is semi-convex.
\end{Prop}
\begin{proof}
Since $\phi$ is $G-$convex, we have $y \in Y$ and $v \in \R$ such that $(x_t, y, v) \in \Psi$ and  
\begin{align*}
& \phi(x_t) = G(x_t, y,v), \\
&\phi(x) \geq G(x, y, v), \ \forall x \in X.
\end{align*}
Moreover, we have
\begin{displaymath}
G(x,y,v) \geq \phi(x_t) + \langle p_t , x-x_t \rangle - \frac{1}{2} \| D^2_{xx} G \| (x-x_t)^2
\end{displaymath}
where $p_t = D_x G(x_t, y, v)$. Evaluate this at $x = x_0$ and $x = x_1$ and add them with weight $(1-t)$ and $t$ respectively.
\begin{align}
\label{semiconvex1}
(1-t) \phi(x_0) + t \phi(x_1) &\geq (1-t) G(x_0 , y, v) + t G(x_1, y, v) \nonumber \\
& \geq \phi(x_t) + \langle p , (1-t)(x_0 - x_t) + t (x_1 - x_t) \rangle \\
& \quad - \frac{1}{2} \| D^2_{xx} G \|  \left( (1-t)(x_0 - x_t)^2 + t (x_1 - x_t)^2 \right). \nonumber
\end{align}
Note that by the choice of $x_t$, we have $(1-t)(x_0 - x_t) + t(x_1 - x_t) = 0$ and
\begin{displaymath}
|x_0 - x_t| = t|x_0 - x_1| \textrm{ and } |x_1 - x_t| = (1-t)|x_0 - x_1|.
\end{displaymath}
Then ($\ref{semiconvex1}$) becomes
\begin{displaymath}
(1-t)\phi(x_0) + t \phi(x_1) \geq \phi(x_t) - \frac{1}{2}t(1-t) \| D^2_{xx} G \| |x_0 - x_1|^2,
\end{displaymath}
which is the desired inequality. This implies that $\phi$ is semi-convex because $t(1-t)|x_0 - x_1|^2 + |x_t|^2 = (1-t)|x_0|^2 + t|x_1|^2$ so that
\begin{align*}
\phi(x_t) + \frac{1}{2} \| D^2_{xx}G \| |x_t|^2 \leq&  (1-t) \phi(x_0) + t \phi(x_1) \\ & + \frac{1}{2} \| D^2_{xx} G \| ( t(1-t) |x_1 - x_0|^2 + |x_t|^2 ) \\
\leq &  (1-t) ( \phi(x_0) + \frac{1}{2}  \| D^2_{xx} G \||x_0|^2 ) \\ & + t ( \phi(x_1) + \frac{1}{2} \| D^2_{xx}G \| |x_1|^2 ) .
\end{align*}
Therefore $\phi(x) + \frac{1}{2} \|D^2_{xx}G \||x|^2$ is convex.
\end{proof}

When we use a norm of some derivatives of $G$ and $H$, we need to check that the points lie in $\h$ and $\g$, and they are in a compact subset of $X \times Y \times \R$. So far, we choose points on the graph of a $G$-convex functions. Then this choice with the $G$-convexity of $G$-subdifferential ensures that the points which we are using are in compact subsets of $\h$ and $\g$. But we need to use other points later (for example, $(x_t, y_i, u)$ in the lemma $\ref{localest}$). Hence we localize the argument and use the next lemma to show our points lie in a compact subset of $\h$.

\begin{Lem}
\label{localniceg}
Let $\phi$ be a $G$-convex function with (nicew) and let $x_0 \in \mathring{X}$. Then there exists $\delta(x_0) >0$ and $S \Subset \h$ such that if $|x_1 - x_0| < \delta(x_0)$, then
\begin{equation}
\label{inS}
\begin{array}{c}
(x_t, \yth , G(x_t,y_0,H(x_0,y_0,\phi(x_0))) \in S \\
(x_t, \yth, \phi(x_t)) \in S
\end{array}
\end{equation}
for any $x_t = (1-t)x_0 + tx_1$, $t \in [0,1]$ and $y_{\theta}$, the $G$-segment connecting $y_0$ and $y_1$ with focus $(x_t, \phi(x_t))$ where $y_0 \in \gsub{\phi}{x_0}$, $y_1 \in \gsub{\phi}{x_1}$.
\end{Lem}
\begin{proof}
Note that by (nicew), we have that $(x_0, y_0, \phi(x_0)) \in \mathring{\h}$ for any $y_0 \in \gsub{\phi}{x_0}$. Therefore, we have $r_1, r_2, r_3 >0$ such that
\begin{equation}
\label{S}
S :=  B_{r_1}(x_0)  \times \left( \nbhd{r_2}{\gsub{\phi}{x_0}} \cap Y \right) \times (\phi(x_0) - r_3 , \phi(x_0) + r_3) \Subset \mathring{\h}
\end{equation}
that is, $\overline{S}$ is compact and $\overline{S}$ is contained in $\mathring{\h}$. Let $C_e$ be the constant from Remark $\ref{nicedom}$, and let $ \gsubs{\phi}{x_0} = \left(\gexp{x_0}{\phi(x_0)} \right)^{-1} \left( \gsub{\phi}{x_0} \right) $. Then by the same remark, we have
\begin{displaymath}
\nbhd{r_2 / C_e}{\gsubs{\phi}{x_0}} \cap \h^*_{x_0, \phi(x_0)} \subset \left( \gexp{x_0}{\phi(x_0)} \right)^{-1} \left( \nbhd{r_2}{\gsub{\phi}{x_0}} \cap Y \right).
\end{displaymath}
Note that $D_x G(x, \cdot, H(x, \cdot, u)) = \left( \gexp{x}{u} \right)^{-1}(\cdot)$. Since the function $(x,y,u) \mapsto D_xG(x,y,H(x,y,u))$ is uniformly continuous on $S$, there exist $\delta_x, \delta_u >0$ such that if $|x - x_0| < \delta_x$ and $|u - \phi(x_0)| < \delta_u$, then
\begin{displaymath}
\left| D_xG(x, y, H(x,y,u)) - D_xG(x_0, y, H(x_0, y, \phi(x_0))) \right| < \frac{r_2}{4C_e}
\end{displaymath}
for any $y \in \nbhd{r_2}{\gsub{\phi}{x_0}} \cap Y$. Hence, for any $y \in \nbhd{r_2}{\gsub{\phi}{x_0}} \cap Y$ such that $ \left( \gexp{x_0}{\phi(x_0)} \right)^{-1} (y) \in \nbhd{r_2/4C_e}{\gsubs{\phi}{x_0}}$, we have
\begin{displaymath}
\left( \gexp{x}{u} \right)^{-1}(y) \in \nbhd{r_2/2C_e}{\gsubs{\phi}{x_0}}
\end{displaymath}
if $|x-x_0| < \delta_x$ and $|u - \phi(x_0)| < \delta_u$. Note that the set $\nbhd{r_2/2C_e}{\gsubs{\phi}{x_0}}$ is convex by $G-$convexity of $G-$subdifferentials. Again from Remark $\ref{nicedom}$, if $y \in \nbhd{r_2/4C_e^2}{\gsub{\phi}{x_0}}$ then $\left( \gexp{x_0}{\phi(x_0)} \right)^{-1}(y) \in \nbhd{r_2/4C_e}{\gsubs{\phi}{x_0}}$. By Proposition $\ref{gsubconti}$, there exists $\delta_1$ such that if $|x-x_0| < \delta_1$, then
\begin{displaymath}
\gsub{\phi}{x} \subset \nbhd{r_2/4C_e^2}{\gsub{\phi}{x_0}}.
\end{displaymath}
Moreover, by continuity of $G$,$H$, and $\phi$, we have $\delta_2$ such that if $|x-x_0| < \delta_2$, then for any $y_0 \in \gsub{\phi}{x_0}$,
\begin{align*}
\allowdisplaybreaks
|G(x,y_0,H(x_0,y_0,\phi(x_0))) - \phi(x_0)| < \min\{\delta_u, r_3\} \\
|\phi(x) - \phi(x_0)| < \min\{ \delta_u, r_3 \}.
\end{align*}
We take $\delta(x_0)$ small enough so that $\delta(x_0) \leq \min \{ \delta_x, \delta_u, \delta_1, \delta_2, r_1, r_3 \}$. Suppose $|x_1 - x_0| < \delta(x_0)$. Then $y_1 \in \gsub{\phi}{x_1} \subset \nbhd{r_2/4C_e^2}{\gsub{\phi}{x_0}}$. Moreover, for any $y_0 \in \gsub{\phi}{x_0}$, we have $\phi(x_t), G(x_t, y_0, H(x_0,y_0,\phi(x_0))) \in ( \phi(x_0) - r_3 , \phi(x_0) + r_3)$. Hence $(x_t, y_0, u)$, $(x_t, y_1, u) \in S \subset \h$ where $u$ is either $\phi(x_t)$ or $G(x_t, y_0, H(x_0,y_0,\phi(x_0)))$. By our construction, the $G$-segment that connects $y_0$ and $y_1$ with focus $(x_t, u)$ is contained in $\nbhd{r_2/2}{\gsub{\phi}{x_0}} \cap Y$. Therefore, we have
\begin{displaymath}
(x_t, \yth, u) \in B_{\delta(x_0)}(x_0) \times \left( \nbhd{r_2/2}{\gsub{\phi}{x_0}} \cap Y \right) \times (\phi(x_0) - r_3 , \phi(x_0) + r_3) \subset S.
\end{displaymath}
\end{proof}

\begin{Rmk}
\label{deponsol}
The constant $\delta(x_0)$ depends on the modulus of continuity of $\phi$ and $\partial_G \phi$ at $x_0$. If we have estimates on these apriori, then we can get rid of this dependency on the solution.
\end{Rmk}
Note that ($\ref{S}$) shows structure of $S$ other than ($\ref{inS}$). We will use this structure of $S$ too. With help of the above lemma, we can bound various norms necessary to obtain the next lemma.

\begin{Lem}
\label{localest}
Let $\phi$ be a $G$-convex function and let $x_0 \in X$. Choose $x_1$ such that $|x_0 - x_1| < \delta(x_0)$. Let $G(x,y_0,v_0)$ and $G(x,y_1,v_1)$ be supporting $G$-affine functions that supports $\phi$ at $x_0$ and $x_1$ respectively. Let $x_t \in [x_0 , x_1]$ such that
\begin{displaymath}
G(x_t, y_0, v_0) = G(x_t,y_1,v_1) =: u.
\end{displaymath}
We assume $|y_1 - y_0| \geq |x_0 - x_1|$, then we have
\begin{equation}
\label{lem4}
\phi(x_t) - u \leq C_3 |x_1 - x_0| |y_1 - y_0|
\end{equation}
where $C_3$ depends on the $C^2$ norm of $G$, hence on constant $C_e$ from Remark $\ref{nicedom}$, which in turn depends on $S$ in Lemma $\ref{localniceg}$
\end{Lem}
\begin{proof}
First of all, the definition of supporting function implies that
\begin{align*}
G(x_0, y_0, v_0) - G(x_0,y_1,v_1) = \phi(x_0) - G(x_0,y_1,v_1) \geq 0 \\
G(x_1, y_0, v_0) - G(x_1, y_1, v_1) = G(x_1, y_0, v_0) - \phi(x_1) \leq 0
\end{align*}
so that existence of $x_t$ is implied by the intermediate value theorem. If $t$ was either 1 or 0, then the left hand side of ($\ref{lem4}$) is 0 so that the lemma is trivial. Otherwise, by our choice of $x_t$ and $u$, we have the equalities
\begin{displaymath}
u = G(x_t, y_0, v_0) = G(x_t, y_0, H(x_0, y_0, \phi(x_0))).
\end{displaymath}
Note that by ($\ref{S}$) of Lemma $\ref{localniceg}$, we have $(x_t, y_i, u) \in S$. We use Taylor expansion on $G(x,y_i,v_i)$ at $x_t$ to get
\begin{equation}
\label{lem4:1}
\phi(x_i) - u \leq \langle D_x G(x_t, y_i, v_i), x_i - x_t \rangle + \frac{1}{2} \| D^2_{xx} G \| |x_i - x_t|^2.
\end{equation}
From Lemma $\ref{semiconvex}$,
\begin{align}
\label{lem4:2}
\phi(x_t) -u & \leq (1-t) ( \phi(x_0)-u) + t (\phi(x_1)-u) + \frac{1}{2} t(1-t) \| D^2_{xx} G \| |x_1 - x_0 |^2 \\
& \leq (1-t) (\phi(x_0)-u) + t( \phi(x_1)-u) + \frac{1}{8} \| D^2_{xx} G \| |x_1 - x_0 |^2. \nonumber
\end{align}
If $(1-t)\langle D_x G(x_t, y_0, v_0), x_0 - x_t \rangle + t (\phi(x_1)-u) \leq 0$, then from ($\ref{lem4:1}$) and ($\ref{lem4:2}$),
\begin{align*}
\phi(x_t) - u & \leq (1-t) \left(  \langle D_x G(x_t, y_0, v_0), x_0 - x_t \rangle  + \frac{1}{2} \| D^2_{xx} G \| |x_0 - x_t|^2 \right) \\
& \ \  +t (\phi(x_1)-u) +\frac{1}{8} \| D^2_{xx}G\| |x_1 - x_0|^2 \\
& \leq (1-t) \frac{1}{2} \| D^2_{xx} G \| |x_0 - x_t|^2 +\frac{1}{8} \| D^2_{xx}G\| |x_1 - x_0|^2 \\
& \leq \frac{5}{8} \|D^2_{xx}G \||x_1 - x_0|^2 \\
& \leq \frac{5}{8}\|D^2_{xx}G \||y_1 - y_0||x_1 - x_0|.
\end{align*}
Otherwise, since $t(1-t) \leq 1$, we have
\begin{align*}
0 & \leq (1-t)\langle D_xG(x_t, y_0, v_0), x_0 - x_t \rangle + t (\phi(x_1)-u) \\
& \leq \frac{1}{t} \langle D_xG(x_t, y_0, v_0), x_0 - x_t \rangle + \frac{1}{1-t} (\phi(x_1)-u) 
\end{align*}
so that 
\begin{align*}
\allowdisplaybreaks
\phi(x_t) -u & \displaystyle \leq  (1-t)\langle D_xG(x_t, y_0, v_0), x_0 - x_t \rangle + t(\phi(x_1) - u) \\
& \displaystyle \ \ + \left( \frac{1}{8}+\frac{1}{2}(1-t)\right)\|D^2_{xx}G \||x_1 - x_0|^2 \\
& \displaystyle \leq  \frac{1}{t} \langle D_xG(x_t, y_0, v_0), x_0 - x_t \rangle + \frac{1}{1-t} (\phi(x_1) - u) \\
&\displaystyle \ \  + \left( \frac{1}{8}+\frac{1}{2}(1-t)\right)\|D^2_{xx}G \||x_1 - x_0|^2 \\
& \displaystyle \leq \frac{1}{t} \langle D_xG(x_t, y_0, v_0), x_0 - x_t \rangle + \frac{1}{1-t} \langle D_xG(x_t, y_1, v_1), x_1 - x_t \rangle \\
& \displaystyle \ \ + \frac{1}{2(1-t)} \|D^2_{xx}G \||x_1 - x_t|^2 + \frac{5}{8}\|D^2_{xx}G \||x_1 - x_0|^2 \\
& \displaystyle = \langle D_xG(x_t, y_0, v_0), x_0 - x_1 \rangle + \langle D_xG(x_t, y_1, v_1), x_1 - x_0 \rangle \\
& \displaystyle \ \ +\frac{1}{2}\|D^2_{xx}G \||x_1 - x_t ||x_1 - x_0| + \frac{5}{8}\|D^2_{xx}G \||x_1 - x_0|^2. \\
\end{align*}
In the last equality, we used $ t = \frac{|x_0 - x_t|}{|x_1 - x_0|}$ and $1-t = \frac{|x_1 - x_t|}{ |x_1 - x_0|}$. We use the fundamental theorem of calculus to obtain
\begin{align*}
\phi(x_t) - u & \displaystyle  \leq \int_0^1 \frac{d}{d\theta} \langle D_x G (x_t, \yth, H(x_t, \yth, u)) , x_1 - x_0 \rangle ds \\ 
&\displaystyle  \ \ \ \ \ + \frac{9}{8}\| D^2_{xx}G \| |x_1 - x_0|^2 \\
&\displaystyle  \leq \| E \| C_e |y_1 - y_0| |x_0 - x_t| + \frac{9}{8}\| D^2_{xx}G \| |x_1 - x_0|^2 \\
& \leq \left(  C_e^2 + \frac{9}{8} \| D^2_{xx}G \| \right) |y_1 - y_0||x_1 - x_0|
\end{align*}
where $\yth$ is the $G$-segment connecting $y_0$ and $y_1$ with focus $(x_t, u)$.
\end{proof}

\begin{Lem}
\label{gsubdiffest}
Let $x_t$ be as in Lemma $\ref{localest}$. There exist $l$, $r$ that depend on $|x_0 - x_1|$ and  $|y_0 - y_1|$ and $\kappa$ such that if $\mathcal{N}_r([x_0, x_1]) \subset X$ and
\begin{equation}
\label{gsubestcond}
|y_0 - y_1| \geq \max\{|x_1 - x_0|, \kappa |x_1 - x_0|^{1/5}\}
\end{equation}
then, choosing $x_1$ close to $x_0$ if necessary, we have
\begin{displaymath}
\mathcal{N}_l \left( \left\{ y_{\theta} | \theta \in \left[\frac{1}{4} , \frac{3}{4}\right] \right\} \right) \cap Y \subset \partial_G \phi ( B_r (x_t) ) 
\end{displaymath}
where $\yth$ is the $G$-segment connecting $y_0$ and $y_1$ with focus $(x_t,u)$ as in the proof of Lemma $\ref{localest}$.
\end{Lem}
\begin{proof}
Let $u$ be as in the proof of Lemma $\ref{localest}$. By the definition of $G$-convexity and Lemma $\ref{qglp}$, we have
\begin{align}
\label{lem5:1}
\phi(x) & \geq \max \{{G(x, y_1, v_0) , G(x,y_1,v_1)} \} \\
& \geq G(x, y_{\theta}, v_{\theta}) + \frac{3}{16}\delta_0|y_1 - y_0|^2|x-x_t|^2 -\gamma |x-x_t|^3 \nonumber
\end{align}
for $\theta \in [ \frac{1}{4} , \frac{3}{4}]$, $|x - x_t| \leq \frac{1}{4}C$ where $v_{\theta} = H(x_t, \yth, u)$. Next, we look at a $G$-affine function that pass through $(x_t , \phi(x_t))$ with focus $(y, H(x_t,y,\phi(x_t)))$ where $y \in \h_{x_t,\phi(x_t)}$, i.e. $G(x, y, H(x_t, y, \phi(x_t)))$. we estimate this function on the boundary of $B_r(x_t)$ to compare with $\phi$. Let $v(y,u) = H(x_t, y, u)$.
\begin{align}
\label{lem5:2}
G(x,y,v(y,\phi(x_t))) & = G(x, y, v(y,\phi(x_t))) - G(x, y_{\theta}, v(\yth, \phi(x_t))) \nonumber \\ 
& \ \ + G(x, y_{\theta}, v(\yth, \phi(x_t))) - G(x, y_{\theta}, v(\yth, u)) \\
& \ \ + G(x,y_{\theta},v(\yth,u)) \nonumber
\end{align}
Note that by Lemma $\ref{localniceg}$, $(x_t, \yth, u) \in S \Subset \h$. For the first line, noting that $\phi(x_t) = G(x_t, y,v(y,\phi(x_t)))$, we have
\begin{align}
\label{lem5:3:1}
\allowdisplaybreaks
&G(x,y,v(y,\phi(x_t))) - \phi(x_t) \nonumber \\
&= \displaystyle \int_0^1 \frac{d}{ds} \left( G(x_t + s(x-x_t), y, v(y, \phi(x_t))) \right) ds \\
&= \displaystyle \int_0^1 \langle D_x G( x_t + s(x-x_t), y, v(y, \phi(x_t))), x-x_t \rangle ds \nonumber
\end{align}
and similar equation holds for $G(x, \yth, v(\yth, \phi(x_t))) - \phi(x_t)$. Therefore, we have
\begin{align}
\label{lem5:3}
\allowdisplaybreaks
&G(x, y, v(y,\phi(x_t))) - G(x, y_{\theta}, v(\yth, \phi(x_t))) \nonumber \\
&=\displaystyle \int_0^1 \langle \left( \begin{array}{l} D_x G( x_t + s(x-x_t), y, v(y, \phi(x_t))) \\ - D_x G( x_t + s(x-x_t), \yth, v(\yth, \phi(x_t))) \end{array} \right),  x-x_t \rangle ds \nonumber \\
&=\displaystyle \int_0^1 \int_0^1 \frac{d}{ds'} \langle \left( D_xG \left(\begin{array}{l} x_t + s(x - x_t), \yth + s'(y-\yth), \\ v( \yth + s'(y - \yth), \phi(x_t)) \end{array} \right) \right), x-x_t \rangle ds'ds \\
&= \displaystyle \int_0^1 \int_0^1 \left( D_{xy}G + D_{xv}G D_y H \right) [y-\yth, x-x_t] ds'ds \nonumber \\
&= \displaystyle \int_0^1 \int_0^1 \left( D_{xy}G + D_{xv}G \frac{D_y G}{-D_v G} \right) [y-\yth, x-x_t] ds'ds \nonumber \\
&\leq C_4 |x - x_t| |y - y_{\theta}|, \nonumber
\end{align}
where $C_4$ depends on the $C^2$ norm of $G$ and $\beta$ (note that the functions in the last integral are evaluated at different points so that $C_4$ might not be equal to $C_e$). For the second line, we use Lemma $\ref{localest}$.
\begin{align}
\label{lem5:4}
&G(x, y_{\theta}, v(\yth, \phi(x_t))) - G(x, y_{\theta}, v(\yth, u)) \nonumber \\
&= \displaystyle \int_0^1 \frac{d}{ds} \left( G{(x, y_{\theta}, v(\yth, u + s(\phi(x_t)-u))} \right) ds \nonumber \\
&= \displaystyle \int_0^1 D_v G D_u H (\phi(x_t)-u) ds \\
&\leq C'_5 |\phi(x_t)-u| \nonumber \\
&\leq C_5|x_1 - x_0| |y_1 - y_0| \nonumber
\end{align}
where $C_5$ depends on the $C^1$ norm of $G$, $\beta$, and $C_3$. applying ($\ref{lem5:3}$) and ($\ref{lem5:4}$) to ($\ref{lem5:2}$),
\begin{equation}
\label{lem5:5}
G(x, y,v(y, \phi(x_t))) \leq G(x,y_{\theta},v(\yth,u)) + C_4 |x - x_t| |y - y_{\theta}| + C_5|x_1 - x_0| |y_1 - y_0|.
\end{equation}
Comparing ($\ref{lem5:1}$) and ($\ref{lem5:5}$), if we have
\begin{equation}
\label{lem5:6}
 C_4 |x - x_t| |y - y_{\theta}| + C_5|x_1 - x_0| |y_1 - y_0| \leq \frac{3}{16}\delta_0|y_1 - y_0|^2|x-x_t|^2 -\gamma |x-x_t|^3,
\end{equation}
then we can obtain $G(x, y, v(y, \phi(x_t))) \leq \phi(x)$. ($\ref{lem5:6}$) is satisfied if we have
\begin{align*}
 C_5|x_1 - x_0| |y_1 - y_0| \leq & \frac{1}{16}\delta_0|y_1 - y_0|^2|x-x_t|^2 \\
 C_4 |x - x_t| |y - y_{\theta}| \leq & \frac{1}{16}\delta_0|y_1 - y_0|^2|x-x_t|^2 \\
 \gamma |x-x_t|^3 \leq & \frac{1}{16}\delta_0|y_1 - y_0|^2|x-x_t|^2.
\end{align*}
Therefore we choose 
\begin{equation}
\label{rlk}
r^2 = \frac{16C_5}{\delta_0} \frac{|x_1 - x_0|}{|y_1 - y_0|}, \quad l = \frac{\delta_0}{16C_4}r |y_1 - y_0|^2, \quad \kappa = \left( \frac{16^3 \gamma^2 C_5}{\delta_0 ^3} \right) ^{\frac{1}{5}}
\end{equation}
so that we obtain
\begin{displaymath}
G(x, y, v(y, \phi(x_t))) \leq \phi(x) \textrm{ for }  y \in \mathcal{N}_l \left( \left\{ y_{\theta} | \theta \in [ \frac{1}{4} , \frac{3}{4} ] \right\} \right) \textrm{ and }  x \in \partial B_r (x_t).
\end{displaymath} Note that $\kappa$ does not depend on $x_0$ and $x_1$. From the condition ($\ref{gsubestcond}$), we know that 
\begin{displaymath}
r^2 \leq \frac{16C_5}{\kappa \delta_0} |x_1 - x_0|^{\frac{4}{5}}, \quad l \leq \frac{\sqrt{\delta_0 C_5}}{4C_4} |x_1 - x_0|^{\frac{1}{2}}\mathrm{diam}(Y)^{\frac{3}{2}}.
\end{displaymath}
Therefore, choosing $x_1$ close enough to $x_0$ so that $|x_1 - x_0| \leq \frac{4C_4^2 r_2^2}{\mathrm{diam}(Y)^{3}\delta_0C_5}$, we can assume that 
\begin{equation}
\label{smallr}
l \leq \frac{r_2}{2} 
\end{equation} 
where $r_2$ is from the proof of Lemma $\ref{localniceg}$. Since $G(x_t, y,v(y, \phi(x_t))) = \phi(x_t)$, we get a local maximum of $G(x, y, v(y, \phi(x_t))) - \phi(x)$ at some point $x_y \in B_r(x_t)$ with non-negative value. Then from the proof of Lemma \ref{localniceg}, we get that $\nbhd{l}{ \left\{ \yth | \theta \in [0,1] \right\} } \cap Y \subset \h_{x_y,\phi(x_t)}$.  If $G(x_y, y, v(y,\phi(x_t))) = \phi(x_y)$, then $G(x, y,v(y, \phi(x_t)))$ is a local support of $\phi$ at $x_y$. Since $\phi$ is $G$-convex, the $G$-affine function $G(x, y, v(y, \phi(x_t)))$ is a global support of $\phi$ and hence $ y \in \partial_G \phi (x_y) \in \partial_G \phi ( B_r (x_t))$ by Proposition \ref{loctoglob} (Recall by Remark $\ref{rmkonunif}$, we can use this proposition). Suppose $G(x_y, y, v(y, \phi(x_t))) > \phi(x_y)$. Note that we have $\phi(x) \geq G(x, y,v(y, u))$. We define 
\begin{align*}
f_y(h) & = \max_{x \in B_r(x_t)} \left\{ G(x, y,v(y, h))  - \phi(x)\right\} \\
& = \max_{x \in B_r(x_t)} \left\{ G(x, y, H(x_t, y, h))  - \phi(x) \right\}
\end{align*}
Then $f_y(\phi(x_t)) > 0$ and $f_y(u) \leq 0$. Moreover, since $\|D_vG\|$ and $\|D_u H\|$ are bounded on $\Psi$ and $\Phi$, the functions in the $\max$ are equicontinuous so that $f_y$ is a continuous function. Therefore, there exists $h_y \in [u, \phi(x_t)]$ at which we have $f_y(h_y) = 0 $. In other words, $G(x, y, v(y, h_y))$ supports $\phi$ at some point $x'$ in $B_r(x_t)$. From ($\ref{smallr}$) and the proof of Lemma $\ref{localniceg}$, we have $(x', y, h_y) \in \h$. Hence we get $y \in \partial_G \phi (B_r(x_t))$. 
\end{proof}

\subsection{Some convex geometry}
In this subsection, we prove a useful convex geometry lemma to estimate the volume of $\nbhd{l}{ \left\{ y_{\theta} | \theta \in \left[\frac{1}{4} , \frac{3}{4}\right] \right\} }$. This subsection corresponds to Lemma 5.10 in \cite{Loeper2009OnTR}. In this paper we  will show the same type of lemma but we do not use the generating function $G$ outside of its domain $X \times Y \times \R$. Note that in \cite{Loeper2009OnTR}, it is used that the cost function $c$ can be extended to and differentiated outside of $\Omega \time \Omega'$.

\begin{Rmk}
\label{convlip}
Suppose we have a compact convex set $A$. Then for each point $p \in \partial A$, there is $r_p > 0$ such that the boundary $\partial A$ can be written as a graph of a convex function up to an isometry in $B_{r_p}(p)$. But since $A$ is compact, we can have $r>0$ that does not depend on $p \in \partial A$ but can replace $r_p$. Moreover, since the convex functions are locally Lipschitz, by taking smaller $r$ if needed, we can assume that each convex function that describes $\partial A$ is a Lipschitz function in $B_r(p)$. Then using compactness again, we can assume that we have a uniform Lipschitz constant. In fact, we can bound the Lipschitz constant by $L = \frac{\textrm{diam{A}}}{r}$. See corollary A.23 of \cite{Figalli2017TheME}.
\end{Rmk}

\begin{Lem}
\label{convvolball}
Let $A$ be a compact convex set. Then there exist $r_A>0$ and $C_A$ that depend on the set $A$ such that for any $r' < r_A$ and $x \in A$, we have
\begin{equation}
\label{volball}
\vol{B_{r'}(x) \cap A} \geq C_A \vol{B_{r'}(x)}.
\end{equation}
\end{Lem}
\begin{proof}
Let $r$ be as in the Remark $\ref{convlip}$, and let $L$ be the Lipschitz constant in the same remark. Let $r' < \frac{r}{2}$. If $x \in \{ q \in A | \dist(q,\partial A) \geq \frac{r}{2} \}$, Then $B_{r'} (x) \subset A$ so that $\vol{B_{r'} (x) \cap A} = \vol{B_{r'} (x)}$. Now suppose $\dist(x, \partial A) < \frac{r}{2}$, then $\exists p \in \partial A$ such that $B_{r'} (x) \subset B_r(p)$. In $B_r(p)$, we have that $\partial A$ is the graph of a convex function $f$ up to an isometry, and $f$ is Lipschitz with the Lipschitz constant $L$. Then, since $x \in A$, we know that $x^n \geq f(x')$ where $x = (x' , x^n) \in \R^{n-1} \times \R$ and we have the inclusions
\begin{align}
\label{inclusionball}
B_{r'} (x) \bigcap  A & \supset B_{r'} (x) \cap \{ q = (q', q^n) \in \R^{n-1} \times \R | q^n \geq f(q') \} ) \nonumber \\
& \supset B_{r'} (x) \cap \{ q | q^n \geq f(x') + L |x' - q'| \}  \\
& \supset B_{r'} (x) \cap \{ q | q^n \geq x^n + L|x' - q'| \}. \nonumber
\end{align}
The last set above is the intersection of a ball centered at $x$ and a cone with vertex at $x$. Hence the volume of this set is a multiple of the volume of the ball where the constant multiplied depends on $L$. Therefore the above inclusion implies ($\ref{volball}$) with $C_A$ depending on $L$, hence on $A$. Therefore the lemma holds with $r_A = \frac{r}{2}$.
\end{proof}

In the next lemma, we use the term ``length of a curve''. The curve that we are using is not necessarily differentiable, so we use next definition for the length of a curve

\begin{Def}
Let $\gamma : [a,b] \to \R^n$ be a continuous curve. We define its length by
\begin{displaymath}
\len{\gamma} = \sup \left\{ \sum_{i=1}^n |\gamma(t_i) - \gamma(t_{i-1})| \big| a = t_0 \leq t_1 \leq \cdots \leq t_n =b \right\} 
\end{displaymath}
\end{Def}
It is well known that this definition preserves a lot of properties of arclength of $C^1$ curves. Some continuous curves may have $\len{\gamma} = \infty$, but $\len{\gamma}$ is finite if $\gamma$ is a Lipschitz curve with finite domain.

\begin{Lem}
\label{convvoltube}
Let $A$ be a compact convex set and let $\gamma : [0,1] \to A$ be a bi-Lipschitz curve, that is
\begin{equation}
\label{bilipcurve}
C_{\gamma}' |s-t| \leq |\gamma(s) - \gamma(t)| \leq C_{\gamma} |s-t|
\end{equation}
for some constants $C_{\gamma}'$ and $C_{\gamma}$. Then there exists $K_A$ and $l_A>0$ that depend on $A$ and $C_{\gamma}'$ such that for any $l \leq l_A$, we have
\begin{equation}
\vol{\nbhd{l}{\gamma} \cap A} \geq K_A C'_{\gamma}l^{n-1}.
\end{equation}
\end{Lem}
\begin{proof}
We assume $l < \frac{r_A}{2}$ where $r_A$ is from the previous lemma $\ref{convvolball}$. Let $m \in \mathbb{N}$ be the smallest number such that $ C_{\gamma} \leq r_A m$. Then by taking $r_A$ smaller if necessary, we have $m \leq \frac{2C_{\gamma}}{r_A}$. Note that we have $\len{\gamma} \leq C_{\gamma}$ from ($\ref{bilipcurve}$). We define $\gamma_i(t) = \gamma( (1-t)\frac{i}{2m} + t \frac{i+1}{2m} )$. Then $\gamma_i$ is bi-Lipschitz with constants $\frac{C_{\gamma}'}{2m}$ for  lower bound and $\frac{C_{\gamma}}{2m}$ for upper bound i.e. 
\begin{displaymath}
\frac{C_{\gamma}'}{2m}|s-t| \leq | \gamma_i(s) - \gamma_i(t) | \leq \frac{C_{\gamma}}{2m} |s-t|.
\end{displaymath} 
In addition, we have $\len{\gamma_i} \leq \frac{r_A}{2}$ by our choice of $m$ and $l$. Suppose $\nbhd{l}{\gamma_i} \cap \partial A \neq \emptyset$. Then we can write $\partial A$ as a graph of a Lipschitz convex function $f$ with Lipschitz constant $L$ around some point $p \in \nbhd{l}{\gamma_i} \cap \partial A$. Moreover, for any $x \in \nbhd{l}{\gamma_i}$, there are some $t, s \in [0,1]$ such that
\begin{align*}
|x-p| & < |x- \gamma_i(t) | + |\gamma_i(t) - \gamma_i(s)| + |\gamma_i(s) - p| \\
& \leq \frac{r_A}{2} + \frac{r_A}{2} + \frac{r_A}{2} = \frac{3}{4}r,
\end{align*}
where $r$ is the radius of the ball around a point on $\partial A$ in which we can write $\partial A$ as a graph of a convex function (Remark $\ref{convlip}$). Therefore $\nbhd{l}{\gamma_i}$ lies in the epigraph of $f$ in $B_r(p)$. Then the proof of Lemma $\ref{convvolball}$ shows that at each point on the curve $\gamma_i$, there exists a conical sector $\mathrm{Sec}_{\gamma_i(t)}$ in $B_l(\gamma_i(t)) \cap A$ which is a translation of the conical sector $\mathrm{Sec_0}$ : 
\begin{displaymath}
\mathrm{Sec}_0 = B_l(0) \cap \{ q | q^n \geq  L|q'| \}.
\end{displaymath}
Note that the inscribed ball in this conical sector has radius $L'l$ where $L'$ is a constant that depends on $L$ so that the inscribed ball is $B_{L'l}(v)$ for some $v \in \mathrm{Sec}_0$. Therefore, for each $0 \leq i \leq 2m-1$, we get $v_i$ such that $\nbhd{L'l}{\gamma_i+v_i} \subset \nbhd{l}{\gamma} \cap A$. Now, if we have $l < \frac{C_{\gamma}'}{4m}$, then for any $x \in \nbhd{l}{\gamma_i}$ and $y \in \nbhd{l}{\gamma_{i+2}}$, we get that for some $s, t \in [0,1]$,
\begin{align*}
|x-y| & \geq |\gamma_i(t) - \gamma_{i+2}(s)| - ( |x-\gamma_i(t)| + | y - \gamma_{i+2}(s)|) \\
& \geq \frac{C_{\gamma}'}{2m} - 2l >0.
\end{align*}
Therefore, $\nbhd{L'l}{\gamma_i} \cap \nbhd{L'l}{\gamma_{i+2}} = \emptyset$. Note that each $\nbhd{L'l}{\gamma_i}$ has volume bounded below by $(L'l)^{n-1}|\gamma_i(0) - \gamma_i(1)| \geq \frac{(L')^{n-1}C_{\gamma}'}{2m} l^{n-1}$ so that 
\begin{align*}
\vol{\nbhd{l}{\gamma} \cap A} & \geq \vol{\bigcup_{i=0}^{2m-1}\nbhd{L'l}{\gamma_i+v_i}} \\
& \geq \vol{\bigcup_{i=0}^{m-1} \nbhd{L'l}{\gamma_{2i}+v_i}} \\
& \geq \frac{(L')^{n-1}C_{\gamma}'}{2m} l^{n-1} \times m = \frac{1}{2}(L')^{n-1}C_{\gamma}'l^{n-1}.
\end{align*}
Therefore we get the lemma with $l_A = \frac{1}{8} \frac{C_{\gamma}'}{C_{\gamma}} r_A \leq \min \{ \frac{r_A}{2} , \frac{C_{\gamma}'}{4m} \} = \frac{C_{\gamma}'}{4m} $ and $K_A = \frac{1}{2}(L')^{n-1}$.
\end{proof}

\begin{Lem}
\label{gsubvolest}
Let $\yth$ be as in Lemma $\ref{gsubdiffest}$ and let $A = \h^*_{x_0, \phi(x_0)}$. If $l \leq \frac{r_A}{8C_e^5}$, then we have
\begin{equation}
\vol{\nbhd{l}{ \left\{ \yth | \theta \in \left[ \frac{1}{4} , \frac{3}{4} \right] \right\} } \cap Y} \geq C_V l^{n-1} |y_0 - y_1|
\end{equation}
where $C_V$ depends on $x_0$, $\h$, and $C_e$.
\end{Lem}
\begin{proof}
Note that $\theta \mapsto \yth$ is a bi-Lipschitz curve with 
\begin{displaymath}
\frac{1}{C_e^2} |y_1-y_0||\theta - \theta'| \leq |\yth - y_{\theta'}| \leq C_e^2 |y_1 - y_0 | |\theta - \theta'|.
\end{displaymath}
Then the reparametrized curve $\theta \mapsto y_{(1-\theta)\frac{1}{4} + \theta \frac{3}{4}}$ is bi-Lipschitz with Lipschitz constants $\frac{2}{C_e^2}|y_1 - y_0|$ and $2C_e^2 |y_1 - y_0|$. Then the curve $\theta \mapsto {\gexp{x_t}{\phi(x_t)}}^{-1}(y_{(1-\theta)\frac{1}{4} + \theta \frac{3}{4}})$ is bi-Lipschitz with Lipschitz constants $\underline{L} = \frac{2}{C_e^3}|y_1 - y_0|$ for lower bound and $\overline{L} = 2C_e^3 |y_1 - y_0|$ for upper bound. Moreover, since the map ${\gexp{x_0}{\phi(x_0)}}^{-1}$ is bi-Lipschitz with Lipschitz constants $\frac{1}{C_e}$ and $C_e$, we have
\begin{align*}
\displaystyle \nbhd{\frac{l}{C_e}}{ {\gexp{x_0}{\phi(x_0)}}^{-1} \left( \left\{ \yth | \theta \in \left[ \frac{1}{4} , \frac{3}{4}  \right] \right\} \right) }  \cap \h^*_{x_0, \phi(x_0)} \\
\displaystyle  \subset {\gexp{x_0}{\phi(x_0)}}^{-1} \left( \nbhd{l}{ \left\{ \yth | \theta \in  \left[ \frac{1}{4} , \frac{3}{4} \right] \right\} } \cap Y \right).
\end{align*}
Note that by (vDomConv), $\h^*_{x_0, \phi(x_0)}$ is convex. Moreover, by our choice of $l$, we have $\frac{l}{C_e} \leq \frac{1}{8} {\underline{L}}/{\overline{L}} r_{A} $. Then by Lemma $\ref{convvoltube}$, we get a constant $K_{x_0}$ that depends on $\h^*_{x_0, \phi(x_0)}$, hence on $x_0$ such that 
\begin{align*}
\displaystyle \vol{\nbhd{\frac{l}{C_e}}{ {\gexp{x_0}{\phi(x_0)}}^{-1} \left( \left\{ \yth | \theta \in \left[ \frac{1}{4} , \frac{3}{4} \right] \right\} \right) } \cap \h^*_{x_0, \phi(x_0)}} \\
\displaystyle \geq K_{x_0} \frac{2}{C_e^3}|y_1 - y_0| (l / C_e)^{n-1}.
\end{align*}
Using bi-Lipschitzness once more, we obtain
\begin{displaymath}
\vol{ \nbhd{l}{ \left\{ \yth | \theta \in  \left[ \frac{1}{4} , \frac{3}{4} \right] \right\} } \cap Y } \geq C_V  l^{n-1} |y_1 - y_0|,
\end{displaymath}
with $C_V = {2K_{x_0}}/{C_e^{2n+2}}$.
\end{proof}

\begin{Rmk}
\label{condforunifLip}
The constant $C_V$ depends on the Lipschitz constant of $\partial \h^*_{x_0, \phi(x_0)}$ so that it depends on $x_0$ and the value of $\phi$ at $x_0$. If we assume that $\{ \h^*_{x,u} \}_{(x,u) \in X \times \R}$ is uniformly Lipschitz, we can get rid of this dependency.
\end{Rmk}

\subsection{Proof of the main theorem}
In the proof of the main theorem, we will use the lemmas in previous subsections. So, we review the conditions to use those lemmas. We choose $x_0, x_1 \in \mathring{X}$ and $y_0, y_1 \in \gsub{\phi}{x_0}$. First of all, to localize the argument, we choose $x_1$ close enough to $x_0$. Explicitly, we choose $|x_0 - x_1|$ smaller than $\delta(x_0)$ to use Lemma $\ref{localniceg}$, smaller than $\frac{4C_4^2 r_2^2}{\mathrm{diam}^3\delta_0C_5}$ to use Lemma $\ref{gsubdiffest}$, and smaller than $\frac{C_4^2r_A^2}{4C_e^{12} \delta_0 C_5} \frac{1}{\mathrm{diam}(Y)^3}$ (with $A = \h^*_{x_0, \phi(x_0)}$) to get the condition in Lemma $\ref{gsubvolest}$. For $y_0$ and $y_1$, we need $|y_0 - y_1| \geq \max{ \{|x_0 - x_1| , \kappa |x_0 - x_1|^{1/5} \}}$ to use Lemma $\ref{gsubdiffest}$. Note that if there is no such $y_0$ and $y_1$, that means we have H$\ddot{\textrm{o}}$lder regularity with exponent $\frac{1}{5}$.

\begin{proof}[Proof of the main theorem $\ref{main}$]
We deal with the first part of the theorem. \\
1. In the first case, we deal with the case $p = \infty$ first. If $p = \infty$, then we have
\begin{displaymath}
\mu \left( B_r(x_t) \right) \leq C \vol{ B_r(x_t) } \leq C' r^n
\end{displaymath}
for some $C$ and $C'$. Moreover, since $\phi$ is an Alexandrov solution, we have
\begin{equation}
\label{main1}
\begin{array}{rl}
\mu \left( B_r(x_t) \right) = \nu \left( \gsub{\phi}{ B_r(x_t)} \right) & \geq \nu \left( \nbhd{l}{ \left\{ \yth | \theta \in \left[ \frac{1}{4} , \frac{3}{4} \right] \right\}} \right) \\
& \geq \lambda C_V l^{n-1}|y_0 - y_1|.
\end{array}
\end{equation}
Combining these, we get $C'r^n \geq C_V l^{n-1}|y_0 - y_1|$. We plug ($\ref{rlk}$) into this inequality to obtain
\begin{displaymath}
|y_0 - y_1| \leq C |x_0 - x_1|^{\frac{1}{4n-1}}
\end{displaymath}
for some constant $C$. Note that this implies single valuedness and H$\mathrm{\ddot{o}}$lder continuity of $\partial_G \phi$.\\
\noindent
Next we deal with the case $p < \infty$. If the condition on $\mu$ holds with $p < \infty$, define $F$ by
\begin{equation}
\label{defineF}
F(V) = \sup\{\mu(B)| B \subset X \textrm{ a ball of volume } V \}.
\end{equation}
Then we have $F(\vol{ B_r (x_t) }) \geq \mu( B_r (x_t) ) = \nu (\gsub{\phi}{B_r (x_t)})$. This with ($\ref{main1}$) implies 
\begin{equation}
\label{boundnuF}
F \left( C \frac{|x_0 - x_1|^{n/2}}{|y_0 - y_1|^{n/2}} \right) \geq C'|x_0 - x_1|^{(n-1)/2}|y_0 - y_1|^{(3n-1)/2}
\end{equation} 
for some constants $C$ and $C'$. From the assumption on $\mu$, we have $F(V) \leq C'' V^{1- 1/p}$ for some $C''$. This with above inequality ($\ref{boundnuF}$), we have
\begin{displaymath}
|y_0 - y_1|^{2n-1+\frac{1}{2}(1-\frac{n}{p})} \leq C|x_1 - x_0|^{\frac{1}{2}(1- \frac{n}{p})}.
\end{displaymath}
Therefore, with the condition $p >n$, we get
\begin{displaymath}
|y_0 - y_1| \leq C |x_0 - x_1|^{\frac{\rho}{4n-2+\rho}}
\end{displaymath}
where $\rho = 1-\frac{n}{p}$. Therefore, we have the following : for any $x_0 \in \mathring{X}$, there exists some constants $r_{x_0}$ and $C_{x_0}$ that depends on $x_0$, $\phi(x_0)$, continuity of $\phi$ at $x_0$ such that if $|x_0 - x_1| < r_{x_0}$, we have 
\begin{equation}
\label{gsubdiffholder}
|y_0 - y_1| \leq C_{x_0} |x_0 - x_1|^{\frac{\rho}{4n-2+\rho}}.
\end{equation}
Then for a compact set $X' \Subset \mathring{X}$, we can cover it with a finite number of balls on which we have ($\ref{gsubdiffholder}$) with respect to the center of the ball. Then we can get the H$\ddot{\textrm{o}}$lder regularity of $\partial_G \phi$ on $X'$ by connecting any two points in $X'$ with a piecewise segment where each segment lies in one of the balls. Then the H$\ddot{\textrm{o}}$lder constant will be bounded by the sum of the H$\ddot{\textrm{o}}$lder constant on each ball that contains a segment times the number of the balls. To get the H$\ddot{\textrm{o}}$lder regularity of the potential $\phi$, we note that $\gsub{\phi}{x} = \gexp{x}{\phi(x)} ( D_x \phi (x) )$, and use remark $\ref{nicedom}$. \\
Now we prove the second part of the theorem.\\
2. Suppose we have $f : \R^{+} \to \R^{+}$ such that $\lim_{r \to 0} f(r) = 0$ and for any $x \in X$ and $r \geq 0$ we have $\mu(B_r(x))\leq f(r)r^{n(1-\frac{1}{n})}$. Note that we can choose $f$ strictly increasing. Then by ($\ref{defineF}$), we have
\begin{equation}
\label{Fbound}
F(V) \leq f\left( \left( \frac{1}{\omega_n} \right) ^{\frac{1}{n}} V^{\frac{1}{n}} \right) \times \left( \frac{1}{\omega_n} \right)^{1 - \frac{1}{n}} V^{1-\frac{1}{n}}
\end{equation}
where $\omega_n$ is the volume of the unit ball in $\R^n$. Define $\tilde{f}$ by 
\begin{displaymath}
\tilde{f}(V) ^{2n-1}= \left( \frac{1}{\omega_n} \right)^{1 - \frac{1}{n}} f \left( \left( \frac{1}{\omega_n} \right)^{\frac{1}{n}} V^{\frac{1}{2}} \right).
\end{displaymath}
Then ($\ref{Fbound}$) becomes
\begin{equation}
\label{Ftildebound}
F(V) \leq \left[ \tilde{f} \left( V^{\frac{2}{n}} \right) \right]^{2n-1} V^{1 - \frac{1}{n}}.
\end{equation}
We combine ($\ref{Ftildebound}$) with ($\ref{boundnuF}$), and we have
\begin{equation}
\label{singlevalue}
\tilde{f} \left( C' \frac{|x_0 - x_1|}{|y_0 - y_1|} \right) \geq C''|y_0 - y_1|
\end{equation}
for some constants $C', C'' >0$. Note that we can assume that $\frac{|x_0 - x_1|}{|y_0 - y_1|} \to 0$ as $|x_0 - x_1| \to 0$ because otherwise, we get a Lipschitz estimate so that we still can get H$\ddot{\textrm{o}}$lder regularity. Then ($\ref{singlevalue}$) implies that $\partial_G \phi$ is a single valued map. Let $g$ be the modulus of continuity of the $G-$subdifferential map $\partial_G \phi$. We divide into two cases. If $g(u) \leq \max \left\{ u, \kappa u^{\frac{1}{5}} \right\}$, then we get $g(u) \to 0$ as $u \to 0$. In the other case, from ($\ref{singlevalue}$) we get
\begin{displaymath}
\tilde{f} \left( C' \frac{u}{g(u)} \right) \geq C'' g(u).
\end{displaymath}
Since $f$ was strictly increasing, so is $\tilde{f}$, so that $\tilde{f}$ is invertible. Therefore the above equation is equivalent to
\begin{displaymath}
u \geq \tilde{f}^{-1} \left( C'' g(u) \right) \frac{g(u)}{C'}.
\end{displaymath}
Let $\omega$ be the inverse of $z \mapsto \tilde{f}^{-1}(C''z) \frac{z}{C'}$. Note that $\omega$ is strictly increasing. Therefore, composing $\omega$ on above inequality shows that
\begin{displaymath}
g(u) \leq \omega (u).
\end{displaymath}
Since the function $z \mapsto \tilde{f}^{-1}(C''z) \frac{z}{C'}$ is strictly increasing and has limit 0 as $ z \to 0$, $\omega(u)$ also has limit 0 as $u \to 0$. Therefore the above inequality implies that $g(u) \to 0$ as $u \to 0$. Hence the modulus of continuity of $\partial_G \phi$ has limit 0 as the variable tends to 0 so that $\partial_G \phi$ is continuous at $x_0$.
\end{proof}

\begin{Rmk}
Note that from Lemma $\ref{localniceg}$ and $\ref{gsubvolest}$, the constants that we get depend on the value of $\phi$ and continuity of $\phi$. Because of these dependencies, the H$\ddot{\textrm{o}}$lder regularity that we get in this paper might not be uniform for solutions to the (GJE). To get a bounds on the H$\ddot{\textrm{o}}$lder norm that do not depend on the solution $\phi$, we need to add some conditions on the set $\h$ so that we can get a uniform Lipschitz constant for $\partial \h^*_{x,u}$ for any $(x,u) \in X \times \R$ as mentioned in Remark $\ref{condforunifLip}$. Moreover, we need an apiori estimate on the modulus of continuity of $\phi$ and $\partial_G \phi$ as mentioned in Remark $\ref{deponsol}$.
\end{Rmk}

\bibliographystyle{amsalpha}

\bibliography{RegJacobiEq.bib}

\end{document}